\documentclass[11pt]{article}
\usepackage[top=2.8cm,bottom=3cm,left=3cm,right=3cm]{geometry}
\usepackage[ansinew]{inputenc}
\usepackage{amssymb,amsfonts,amsmath,amsthm,mathabx,mathrsfs,color,indentfirst}
\usepackage{graphicx}

\makeatletter
\let\@fnsymbol\@arabic
\makeatother 

\renewcommand{\theequation}{\arabic{section}.\arabic{equation}}
\newcommand{\tphi}{\tilde{\varphi}}
\newcommand{\tpsi}{\tilde{\psi}}
\newcommand{\tw}{\tilde{w}}  
\newcommand{\Z}{\textbf{Z}}
\newcommand{\ls}{\leqslant}

\newcommand{\<}{\left\langle}
\newcommand{\PI}{\right\rangle}

\newcommand{\dt}{\mbox{d}t}
\newcommand{\dx}{\mbox{d}x}

\newtheorem{Theorem}{Theorem}[section]
\newtheorem{Lemma}{Lemma}[section]

\newtheorem{Remark}{Remark}[section]
\newtheorem{Definition}{Definition}[section]
\newtheorem{Proposition}{Proposition}[section]

\begin{document}
	
\title{\Large \bf Attractors for locally damped Bresse systems and a unique continuation property}

\author{\normalsize 
To Fu Ma \\ {\small Department of Mathematics, University of Bras\'{\i}lia, 70910-900 Bras\'{\i}lia, DF, Brazil}  \\ \\ 
Rodrigo N. Monteiro \thanks{Corresponding author.} \\ {\small Department of Mathematics, State University of Londrina, 86057-970 Londrina, PR, Brazil} \\ \\
Paulo N. Seminario-Huertas \\ {\small Department of Mathematics, University of Bras\'{\i}lia , 70910-900 Bras\'{\i}lia, DF, Brazil} }

\date{}

\maketitle

\begin{abstract} 
This paper is devoted to Bresse systems, a robust model for circular beams, given by a set of three coupled wave equations. 
The main objective is to establish the existence of global attractors for dynamics of semilinear problems with localized damping. In order to deal with localized damping a unique continuation property (UCP) is needed. Therefore we also provide a suitable UCP for Bresse systems. Our strategy is to set the problem in a Riemannian geometry framework and see the system as a single equation with different Riemann metrics. Then we perform Carleman-type estimates to get our result. 
\end{abstract}



\tableofcontents

\section{Introduction}

The Bresse system is a model for circular beams given by three coupled wave equations, namely, 
\begin{align*}
\left\{ 
\begin{array}{rrr}
\rho_{1}\varphi_{tt} - k(\varphi_{x}+\psi+\ell w)_{x} - k_{0}\ell (w_{x}-\ell\varphi) = 0  \;\;  \mbox{in} \:\: (0,L) \times (0,\infty) , \\
\rho_{2}\psi_{tt} - b\psi_{xx}+k(\varphi_{x} + \psi+\ell w) = 0   \;\;  \mbox{in} \: \:(0,L) \times (0,\infty),   \\
\rho_{1}w_{tt} - k_{0}(w_{x}-\ell\varphi)_{x} + k\ell(\varphi_{x}+\psi+\ell w) =0  \;\; \mbox{in} \: \:(0,L) \times (0,\infty). 
\end{array}
\right.
\end{align*}
The functions $\varphi =\varphi(x,t)$, $\psi=\psi(x,t)$, $w = w(x,t)$ correspond to the vertical displacement, shear angle and longitudinal displacement at a point $x \in  (0,L)$ and time $t\geqslant 0$, respectively. The coefficients are all positive constants defined by   
$\rho_{1}= \rho A$, $\rho_{2}=\rho I$, $k= K A G$, $b= EI$ and $k_0 = AE$, 
where the quantities $\rho$, $A$, $I$, $K$, $G$ and $E$ denote respectively, material density, cross-sectional area, 
second moment of the cross-section area, a shear factor, shear modulus and modulus of elasticity. 
In addition, $\ell >0$ denotes the beam's curvature. 
Its mathematical modeling can be found in \cite{Bresse,lagnese-book} and it is worth observing that when $\ell =0$ the arched beam reduces to the Timoshenko beam \cite{Timo}. 

The Bresse system and its viscoelastic and thermoelastic extensions were studied by many authors. 
Roughly speaking, most of results are concerned with a certain asymptotic stability dichotomy. Indeed, 
by analogy to the Timoshenko system, the Bresse system having damping terms in one or two of its equations is exponentially stable if and only if satisfies the equal wave speeds condition  
\begin{equation} \label{equal-speed}
\frac{\rho_1}{k} = \frac{\rho_2}{b} \quad \mbox{and} \quad k=k_0. 
\end{equation}
Such a condition was firstly observed for Timoshenko systems in \cite{Sou}. Otherwise only polynomial stability can be obtained. See e.g. \cite{alabau,delloro,fatori-monteiro,fatori-munoz,liu-rao,santos-soufyane,soufyane-said,wehbe}. 
 
In a different direction, Charles et al \cite{charles-soriano} proved the exponential stability of Bresse systems 
by adding a localized damping in each one of its three equations, without assuming the speed condition \eqref{equal-speed}. 
This is quite interesting since the equal speed assumption can not be realized physically, cf. \cite{olsson}. 

The main objective of this paper is to establish existence of global attractors for dynamics of a semilinear Bresse system with locally defined damping (see problem \ref{P01}), without assuming condition \eqref{equal-speed}. Our approach is very different from the above one in \cite{charles-soriano}. Indeed, one of ingredients for obtaining exponential stability of wave equations with localized damping is a unique continuation property (UCP). To our purpose, the UCP says whether a wave equation that vanishes in a subdomain must be identically null. In \cite{charles-soriano} they have used a UCP derived from Holmgren uniqueness theorem, which is only valid for equations with analytic coefficients. Because of nonlinear terms, our problem \eqref{P01} has no longer analytic coefficients. To overcome this difficult we propose a new UCP for Bresse systems. 

More precisely, we discuss the unique continuation property for coupled wave equations of the form 
\begin{align} \label{P}
\left\{ 
\begin{array}{lll}
\begin{split}
\partial^2_{t} u_1-\Delta_1 u_1=f_1(u_1,u_2,\cdots,u_n)&  \,\,\, \mbox{in} \,\,\, (0,L)\times(0,\infty),\\
\partial^2_{t} u_2-\Delta_2 u_2=f_2(u_1,u_2,\cdots,u_n)&  \,\,\, \mbox{in} \,\,\, (0,L)\times(0,\infty),\\
&\,\,\,\,\,\vdots\\
\partial^2_{t} u_n\!-\Delta_n u_n=f_n(u_1,u_2,\cdots,u_n)& \,\,\, \mbox{in} \,\,\, (0,L)\times(0,\infty),
\end{split}
\end{array}
\right.
\end{align}
where for $i=1,\cdots,n$ the following is assumed:
\begin{enumerate}
	\item We consider on the system \eqref{P} the Dirichlet boundary conditions
	\begin{align}\label{boundary}
	u_{i}(0,t)=u_{i}(L,t)= 0 \,\,\,\, \forall t \in (0,\infty),
	\end{align}
	and initial data
	\begin{align}\label{initial}
	(u_{1},\partial_{t}u_{1},\cdots, u_{n},\partial_{t}u_{n})\bigl|_{t=0}=(u_{1}^{0},u_{1}^{1},\cdots,u_{n}^{0},u_{n}^{1})
	\,\,\mbox{in}\,\,(0,L).
	\end{align}
	
	\item Given $\gamma_i>0$, the operator $\Delta_i$ represents the one-dimensional Laplacian operator with wave propagation velocity $\sqrt{\gamma_i}$, defined by
	\begin{equation} \label{delta-i}
	\Delta_i=\gamma_{i}\partial^2_{x}.
	\end{equation}
	
	\item The symbols $f_{i}$ denote the coupling functions with energy level terms such that $f_i \in L^2(0,T;L^2(0,L))$ and
	\begin{equation} \label{f_i}
	f_i(u_1,u_2,\cdots,u_n)=\sum^n_{j=1}p^i_j \partial_x u_j+\sum^n_{j=1}q^i_j u_j,
	\end{equation}
	where $p^i_j,q^i_j \in L^2(0,T;L^2(0,L))$.
	Additionally, given $T>0$ there exists a constant $C_T>0$ such that
	\begin{equation} \label{F}
	\int_{0}^{T}\int_{0}^{L}|f_i(u_1,u_2,\cdots,u_n)|^2 \dx\dt \le C_T \int_{0}^{T} F_u(t)\dt,
	\end{equation}
	where $F_u(t)$ represents the energy of the system \eqref{P} defined by
	\begin{equation} \label{E}
	F_u(t)=\sum_{i=1}^{n}F_{u,i}(t)
	\end{equation}
	and $F_{u,i}(t)$ is the energy of the $i$-th equation of the system given by
	\begin{equation} \label{Ei}
	F_{u,i}(t)=\int_{0}^{L}\Bigl[|u_i|^2+ \gamma_{i}|\partial_x u_{i}|^2+|\partial_t u_i|^2 \Bigl]\dx.
	\end{equation}
\end{enumerate} 

The UCP has been extensively used in the analysis of exact controllability, exponential stability, and in the theory of attractors for locally damped wave equations.
The stabilization problem for linear wave equations on a smooth compact Riemannian manifold was studied, e.g., in \cite{blr, rt-cpam, rt-indiana}. 
In these papers, to show the exponential decay rates for the energy the authors assume localized damping and their proof uses a UCP for the wave equation based on Holmgren Theorem. In \cite{TAMS, ARMA} the authors treated the nonlinear case exhibiting an exponential decay of the energy with \textit{sharp} damping region, roughly speaking, a damping region with  arbitrarily small measure. In this case, a new UCP is proven by means of energy estimates and a escape vector field based on \cite{L-Magenes}. 

The study of the existence of global attractors for the wave equation with external forces of critical exponent and locally distributed damping has been established in \cite{Pau, clt, fz, Ma-Se}. Despite dealing with wave equations, the significant difference between these models is the damping regions imposed on the damping parameter. Therefore different types of UCP are needed. Reference \cite{fz} uses the UCP given in \cite{ruiz}. 
In \cite{clt, Ma-Se}, the authors apply the UCP corresponding to Carleman-type estimates for wave equations with linear potential and in \cite{Pau} the authors introduced a new UCP using the techniques of \cite{Triggiani}. 

Here, the main idea for proving a new UCP for Bresse systems is to set Problem \eqref{P} in a 
Riemannian geometry framework and see the system as a single equation with different Riemann metrics.   
Then we show how Carleman-type estimates obtained in \cite{Triggiani} 
can be used to obtain a UCP for our system $(\ref{P})$ under above assumptions 
on the functions $f^{i}(\cdot)$ $(i=1,2,\cdots,n)$. 

Our paper is organized in the following way. Our UCP - Theorem \ref{mainUCP} - will follow after a series of comparison results between reference \cite{Triggiani}. $(i)$ We begin in Section 2 with a Riemannian geometry background material.  
$(ii)$ After, in Section 3, we introduce a preliminary material that will lead the corresponding Carleman-type estimates for the Problem $(\ref{P})$. Finally, as a consequence of Carleman-type estimates, we then achieve our goal, the proof of the Theorem \ref{mainUCP}. For completeness, to the best of our knowledge, this is the first UCP result for coupled wave equations. 

In the second part of the paper, we establish the existence of global attractors for the Bresse system with a nonlinear foundation 
and nonlinear localized dissipation - see Problem \eqref{P01}. We note that in \cite{MaMonteiro} the authors studied a Bresse system with nonlinear foundation and dissipation acting on the whole domain. There, UCP and observability inequalities were not necessary. In this sense, our application improves the previous results on the existence of long-time dynamics 
of the Bresse system allowing the dissipation to be localized in an arbitrary subset of $(0, L)$. 

The outline of the remainder of the paper is the following: 
$(iii)$ In Section 4, we introduce the semilinear Bresse system with localized dissipation along with the well-posedness result and energy estimates. 
The main result is the Theorem \ref{Main} and whose proof is based on the following strategy: 
$(a)$ we first show the existence of a strictly Lyapunov function for the associated dynamical system by using the new UCP stated in  Theorem \ref{mainUCP} and 
$(b)$ introducing observability inequalities, we prove the asymptotic smoothness of the problem using the abstracts results on the recent theory of quasi-stable systems \cite{Yellow}. 
Here, we also mention the importance of the UCP for the proof of a strictly Lyapunov function - see Definition \ref{Lya}.  
$(iv)$ The Appendix is devoted to the well-posedness result for wave equations with over-determined conditions.

\section{A Riemannian geometry framework}

\subsection{Basic notation}
Let $(M,g)$ be an $n$-dimensional, compact Riemannian manifold, with smooth boundary and smooth metric. 
The tangent space on $M$ at $p$ is denoted by $T_p M$ and fix a coordinate system $(x_1,\cdots,x_n)$ then $(\partial_{x_1},\cdots,\partial_{x_n})$ represents the associated coordinate vector fields. In this case
$$
g(X,Y)=\<X,Y\PI=\sum_{i,j=1}^{n}g_{ij}\alpha_i \beta_j, \ \ |X|^2=\<X,X\PI,
$$
where
\begin{equation} \label{XY}
X=\sum_{i=1}^{n}\alpha_i \partial_{x_i}, \ \ Y=\sum_{i=1}^{n}\beta_i \partial_{x_i} \ \ {\text{in $T_pM$ for some $p \in M$}},
\end{equation}
and
$$
g_{ij}=\<\partial_{x_i},\partial_{x_j}\PI.
$$
Note that $|\cdot|$ represents the norm with respect to the metric $g(\cdot,\cdot)$. 
In particular, we denote the inner product $g(\cdot, \cdot)$ by the matrix  $(g_{ij})_{n\times n}$ and its inverse by $(g^{ij})_{n \times n}$. 

\noindent The tangent and cotangent bundle of $M$ are respectively detonate by $TM$ and $T^*M$. 
The symbol $D$ denotes the Levi-Civita connection of $M$ such that for two vector fields $X$ and $Y$ on $M$ given by (\ref{XY}) the following equality hods true
$$
D_X Y=\sum_{i,k=1}^{n}\left[ \alpha_i \partial_{x_i} \beta_k \partial_{x_k}+\sum_{j=1}^{n}\beta_j \Gamma^k_{ij}\partial_{x_k}\right],
$$
where $\Gamma^{k}_{ij}$ represent the Christoffel symbols.

\noindent Let $f:M \to \mathbb{R}$ and $H \in T_p M$ for all $p \in M$. 
\begin{enumerate}
	\item If $f \in C^1(M)$ then the differential $D f:TM \to \mathbb{R}$ represents the gradient of the connection $D$ on $f$ and 
	$$
	D f(H)=D_H f=H(f)=\<\nabla f,H\PI,
	$$
	where $\nabla$ is the usual gradient defined in a coordinate system by
	\begin{equation} \label{gra}
	\nabla f=\sum^{n}_{i,j=1} g^{ij}\partial_{x_i}f\partial_{x_j}.
	\end{equation}

Thanks to the musical isomorphism we will identify $Df$ with $\nabla f$. Here, we often denote $Df$ by $\nabla f$. 
In particular, if $\{E_1,\cdots,E_n\}$ represents an orthonormal basis of $T_p M$ and $H=\sum_{i=1}^{n}h_i E_i$ then
	$$
	Df(H)=H(f)=\sum_{i=1}^{n}h_i E_i(f).
	$$
	\item  If $f \in C^2(M)$ then $D^2 f$ represents the Hessian of $f$ such that for all $Y \in TM$
	$$
	D^2 f(\cdot,Y)=D(D f)(\cdot,Y)=D_Y(\nabla f(\cdot)):TM \to \mathbb{R},
	$$
 with
	$$
	D_Y(\nabla f(X))=\<D_X(\nabla f),Y\PI, \ \ \forall X \in TM.
	$$
	In particular
	$$
	D^2 f(X,X)=\<D_X(\nabla f),X\PI, \ \ \forall X \in TM.
	$$
	
	\item Let $f \in C^3(M)$.  The function $f$ is strictly convex in the metric $g$ if and only if $D^2 f(X,X)>0$ for all $X \in TM$.
	
	\item If $\{E_1,\cdots,E_n\}$ represents an orthonormal basis of $T_pM$, then the divergent of $H$ is defined as
	$$
	{\rm div}(H)=\sum_{i=1}^{n}\<D_{E_i} H,E_i\PI.
	$$
	
	\item If $f \in C^1(M)$ then
	$$
	{\rm div}(fH)=f {\rm div}(H)+H(f).
	$$
	
	\item For a function $f \in C^2(M)$ we define the Laplace-Beltrami operator $\Delta$ by
	$$
	\Delta f={\rm div}(\nabla f)=\sum_{i=1}^{n}\<D_{E_i}(\nabla f),E_i\PI=\sum_{i=1}^{n}D^2 f(E_i,E_i).
	$$ 
	
	\item  The covariant derivate $DH$ is the bilinear form given by
	$$
	DH(X,Y)=\<D_X H,Y\PI, \ \ \forall X,Y \in TM.
	$$
	In particular if $f \in C^2(M)$
	$$
	D(\nabla f)(X,Y)=D^2 f(X,Y), \ \ \forall X,Y \in TM.
	$$
	
	\item	If $f \in C^1(M)$ then
	$$
	\<\nabla f,\nabla(H(f))\PI=DH(\nabla f,\nabla f)+\frac{1}{2}{\rm div}(|\nabla f|^2 H)-\frac{1}{2}|\nabla f|^2{\rm div}H.
	$$
\end{enumerate}

\begin{Remark}
	Let $\Omega$ an open bounded, connected, compact subset of $M$ with smooth boundary $\partial \Omega$ 
	and $f \in C^3(M)$ a strictly convex function in the metric $g$. 
	Then by translating and rescaling {\rm \cite[Remark 1.2]{Triggiani}}, the function $f$ satisfies the following conditions
	\begin{equation} \begin{split}\label{stric} 
	& D^2f(X,X)\ge |X|^2, \ \forall p \in \Omega, \ \forall X \in T_p M, \\
	& \min_{\overline{\Omega}}f(x) \equiv m>0.  
	\end{split}
	\end{equation}
\end{Remark}

\subsection{Geometry on the wave system}

Let us consider $\mathbb{R}$ with usual topology and $x$ the natural coordinate system. 
In particular, for each $x \in \mathbb{R}$ the tangent space is $T_x\mathbb{R}=\mathbb{R}_x=\mathbb{R}$. 

\noindent For a fixed  $i \in \{1,\cdots,n\}$, we begin defining the metrics $g_{i}$ associated with $\Delta_i$ by
\begin{equation}\label{gi}
g_i(X,Y)=\<X,Y\PI_{g_{i}}=\gamma^{-1}_{i}\alpha\beta,
\end{equation}
with corresponding norm
$$
|X|_{g_i}^{2}=\<X,X\PI_{g_i},
$$
where $X=\alpha \partial_{x},\, Y=\beta\partial_{x} \in \mathbb{R}_x,$ for each $x \in \mathbb{R}$.

\noindent Recalling assumption $\gamma_i>0$, one can see that the pairs $(\mathbb{R},g_i)$ are Riemannian manifolds. 
In this case, the Levi-Civita connection of $(\mathbb{R},g_i)$ will be denoted by $D_{g_i}$. 
Here, the symbol $(\mathbb{R},g_0)$ denotes the $\mathbb{R}$ space with Euclidean metric and we use the following notations for the metric and norm
\begin{equation}\label{g0}
g_0(X,Y)=X\cdot Y\,\,\,{\rm and}\,\,\,|X|_{g_0}=|X|, \ \ \forall X, Y \in \mathbb{R}_x, \ \ \forall x \in \mathbb{R}.
\end{equation}

Important properties of the above metrics are stated in the following lemma. 
Although most of these results are followed straightforwardly from the known results, they are crucial for what follows. 
So for the convenience of the reader, we give their proofs here.

\begin{Lemma}\label{geo} Let $(0,L) \subset \mathbb{R}$ for some $L>0$. 
If $x$ be the natural coordinate system in $\mathbb{R}$, $f,h \in C^1([0,L])$ and  $X=\alpha \partial_x, Y=\beta \partial_x$ vector fields. Then
	\begin{enumerate}
		\item [$1.$] $\<X,\gamma_{i} Y\PI_{g_i}=X\cdot Y$;
		
		\item [$2.$] $\nabla_{g_i}f=\gamma_{i} \partial_x f \partial_x$;
		
		\item [$3.$] $X(f)=\<\nabla_{g_i}f,X\PI_{g_i}=\alpha \partial_x f$;
		
		\item [$4.$] $\<\nabla_{g_i}f,\nabla_{g_i}h\PI_{g_i}= \gamma_i \partial_x f \partial_x h$;
		
		\item [$5.$] if $\nu$ represents the unit outward normal vector for $([0,L],g_i)$ then
		$$
		\<\nabla_{g_i} f(0),\nu(0)\PI_{g_i}=-\sqrt{\gamma_{i}} \partial_x f(0), \ \ \<\nabla_{g_i} f(L),\nu(L)\PI_{g_i}=\sqrt{\gamma_{i}} \partial_x f(L).
		$$
	\end{enumerate}
	
\end{Lemma}

\begin{proof}
	In light of \eqref{gi} and \eqref{g0}, one can easily see that Item $1$ holds true. In fact, by definition 
	\begin{equation} \label{x}
	\<X,\gamma_{i} Y\PI_{g_i}=\gamma^{-1}_i\gamma_i\alpha\beta=X\cdot Y.
	\end{equation}
To prove Item 2, we recall the definition \eqref{gra} to find
	$$
	\nabla_{g_i}f=\gamma_i \partial_x f \partial_x.
	$$
	Reasoning analogously to \eqref{x} and assuming $h \in C^1([0,L])$, we have the validity of Items $2$-$4$. To conclude, let us show Item $5$. Firstly, note that if $p \in \{0,L\}$ and $\{\partial_x\}$ represents the associated coordinate vector field for $T_p[0,L]$, then
	$$
	\<\pm\sqrt{\gamma_{i}}\partial_x,\pm\sqrt{\gamma_{i}}\partial_x\PI_{g_i}=1.
	$$
 In particular $E=\pm\sqrt{\gamma_{i}} \partial_x$ represents an orthonormal basis for $p \in T_pM$. Let $\nu=E$, therefore
	$$
	\nabla_{g_i} f=\gamma_{i} \partial_x f \partial_x=\pm\sqrt{\gamma_{i}}\partial_x f\nu
	$$
	and 
	$$
	\<\nabla_{g_i} f,\nu\PI_{g_i}=\pm \sqrt{\gamma_{i}}\partial_x f.
	$$
	Note that $\nu(0)=-\sqrt{\gamma_{i}}\partial_x$ and $\nu(L)=\sqrt{\gamma_{i}}\partial_x$. We complete the proof of	Lemma.
\end{proof}

\begin{Remark}
	The Hessian of $f \in C^2(\mathbb{R})$ with respect to the metric $g_i$ is given by 
	$$
	D_{g_i}^2f(X,X)=\<D_{g_i,X} (\nabla_{g_i}f),X\PI_{g_i}=\alpha^2\gamma_{i}^{-1}\gamma_{i}\partial^2_x f=\alpha^2\partial^2_x f,
	$$
	where $X=\alpha \partial_x.$
	We observe that the Hessian of $f$ is positive if and only if $\partial^2_x f$ is positive.
\end{Remark}

\section{Unique continuation property (UCP)}

As already observed Carleman estimates are an important tool for proving unique continuation property 
for solutions to partial differential equations \cite{Pau, LTZ, Triggiani}. 
In this section, we prove the Carleman estimates for the Problem $(\ref{P})$.  

In what follows we shall use the following notations.
\begin{enumerate}
	\item Let $L_0$ a real number in $[0,L]$. We define $\Omega_1$ and $\Omega_2$ the subsets of $[0,L]$ as follows 
	\begin{equation} \label{omejaj}
	\Omega_1=(0,L_0)\,\,\mbox{and}\,\,\Omega_2=(L_0,L).
	\end{equation}
	Note that $\overline{\Omega}_1 \cup \overline{\Omega}_2=[0,L]$.
	
	\noindent For $\varepsilon>0$, we will also consider the following subsets of  $\Omega_1$ and $\Omega_2$
	$$
	V_1=\left(0,L_0-\frac{\varepsilon}{4}\right)\!\cap\! [0,L], 
	\ V_2=\left( L_0+\frac{\varepsilon}{4},L\right)\!\cap\! [0,L], 
	\  \omega=\left(L_0-\frac{\varepsilon}{2},L_0+\frac{\varepsilon}{2}\right) \!\cap\! [0,L].
	$$
	
	\item There exist strictly convex functions $d_1:[0,L]\to \mathbb{R}$ and $d_2:[0,L]\to \mathbb{R}$ such that
	\begin{equation} \label{dj}
	d_1(x)=\frac{1}{2}(x+L)^2, \ \ \text{and} \ \ d_2(x)=\frac{1}{2}(x-2L)^2.
	\end{equation}
	In particular, if $x$ denotes the natural coordinate system, Lemma \ref{geo} implies that
	$$
	\nabla_{g_i} d_1=\gamma_i (x+L)\partial_x,  \ 	\nabla_{g_i} d_2=\gamma_i (x-2L)\partial_x \ \ (i=1,\cdots,n),
	$$
	and
	$$
	D_{g_i}^2d_j(X,X)=\alpha^2, \  \mbox{for all}\  x \in [0,L]  \ \mbox{and}  \ X=\alpha\partial_x \in T_x[0,L] \ \ (j=1,2).
	$$
\end{enumerate}

The functions $d_j(\cdot)$ $(j=1,2)$ have the following properties: 
\begin{Lemma}\label{lemaover}
Under the above definitions, the functions $d_j(\cdot)$ $(j=1,2)$ satisfy
\begin{enumerate}
		\item[$1.$] $d_j \in C^\infty([0,L])$;
		
		\item[$2.$] $D^2_{g_i}d_j(X,X)=|X|_{g_i}, \ \mbox{for all} \ x \in \Omega_j \ \mbox{and}\ X \in T_x[0,L]$;
		
		\item[$3.$]$\displaystyle{\inf_{\Omega_j}|\nabla_{g_i} d_j|_{g_i}>0}$;
		
		\item[$4.$]$\displaystyle{\min_{\overline{\Omega}_j}d_j=\frac{L^2}{2}>0}$.
	\end{enumerate}
Moreover, if $\nu$ represents the unit outward normal vector then
	\begin{enumerate}
		\item[$5.$] $\<\nabla_{g_i}d_j(x),\nu(x)\PI_{g_i}<0$ on $\{0,L\} \cap \overline{V_j}$.
	\end{enumerate}
\end{Lemma}

\medskip

The next section is devoted to proof the Carleman estimate compatible with the system \eqref{P}. 
To this aim, we allocated the above notations in the same context as Section 1 in \cite{Triggiani}. 
First, without loss of generality (by rescaling), we can assume that 
\begin{equation}
k_{ij}\equiv \inf_{x \in \Omega_j}\frac{|\nabla_{g_i} d_j|^2_{g_i}}{d_j}>4, \ \ \forall i=1,\cdots,n, \ \ \forall j=1,2.
\end{equation}

\noindent In what follows, for fixed $i \in \{1,\cdots,n\}$ and $j=1,2$, we define
\begin{equation} \label{T0}
T_{0,j}^{2} \equiv 4\max_{x \in \overline{\Omega}_j} d_j(x). 
\end{equation}
Let $T>\max\{T_{0,1},T_{0,2}\}$ and by $(\ref{T0})$ there exist $\delta>0$ and $c =c_{\delta} \in (0,1)$ satisfying
\begin{equation} \label{c}
cT^2>4\max_{x \in \overline{\Omega}_j}d_j(x)+4\delta.
\end{equation}

\noindent In above context, we define functions $\phi_j:\Omega_j\times \mathbb{R}\to \mathbb{R} \in C^3(\Omega_j)$ by 
	$$
	\phi_j(x,t)\equiv d_j(x)-c\left(t-\frac{T}{2}\right)^2,\ \ (x,t) \in \Omega_j\times [0,T].
	$$ 
The following properties are valid for $\phi_j(\cdot)$:
\begin{enumerate}
	\item[$(\phi.1)$] For the constant $\delta>0$
	$$
	\phi_j(x,0)=\phi_j(x,T)=d_j(x)-c\frac{T^2}{4}\le \max_{j=\{1,2\}}\left(\max_{x \in \overline{\Omega}_j}d_j(x)\right) -c\frac{T^2}{4}\le -\delta,
	$$
	uniformly in $\overline{\Omega}_j$.
	\item[$(\phi.2)$] There are $t_0$ and $t_1$ with $0<t_0<T/2<t_1<T$, such that
\begin{equation} \label{sigma-min}
	\min_{j=1,2}\left(\min_{(x,t) \in \overline{\Omega}_j\times [t_0,t_1]} \phi_j(x,t)\right)  \ge \sigma,
\end{equation} for $\sigma \in \left(0, \min_{\overline{\Omega}_j}d_j\right)$, where $\displaystyle{\min_{\overline{\Omega}_j}d_j} = \frac{L^2}{2}$ 
(Lemma $\ref{lemaover}$).
\end{enumerate}

As a consequence of Corollary 4.2 in \cite{Triggiani}
\begin{equation} \label{es}
\max_{i=1,\cdots, n \ \ j=1,2}\left(\max_{(x,t) \in \overline{\Omega}_j\times [0,T]} \left(|\partial_t \phi_j|^2+|\nabla_{g_i} \phi_j|_{g_i}^2 \right)\right) \le \frac{4(1+7c)\sigma^*}{\varepsilon(1-c)},
\end{equation}
for any $\varepsilon \in (0,\min\{2-2c,1\})$ with $\sigma^* \in (0,\sigma)$ and $c \in (0,1)$.

We end this section defining 
\begin{equation} \label{Q}
	Q(\sigma)=\left\lbrace (x,t) \in [0,L]\times [0,T] \ \left| \right.  \ \min_{j=1,2}\phi_j(x,t)\ge \sigma\right\rbrace.
\end{equation}
This set will play an important role in Carleman estimates by being able to separate the set $[0,L]\times [0,T]$ 
from the level surface generated by the pseudo-convex function $\phi_j$ at height of $\sigma$.

\noindent Additionally, note that
$$
\Omega_j \times [t_0,t_1]\subset Q(\sigma) \subset [0,L]\times [0,T], \ \ \forall j=1,2.
$$


\subsection{Analysis of the coupled system - Problem \eqref{P} }

In this section, we will study the Problem \eqref{P} under the new decomposition $\Omega_1$ and $\Omega_2$ (\ref{omejaj}).
This decomposition will allow us to define the boundary terms for the solutions of this system. We begin with the following definition:

\begin{Definition}\noindent Let $T>0$. 
The vector function $(u_1,\cdots,u_n)$ is a weak solution to the Problem \eqref{P}-\eqref{initial} 
if the function $u_i$ solves the variational form of equation $\eqref{P}_{i}$ and possesses regularity
$$
u_i \in C(0,T;(H^1_0(M_i)))\cap C^1(0,T;(L^2(M_i))),
$$
where $M_i$ denotes the Riemannian manifold $([0,L],g_i)$ $(i=1,\cdots,n)$.

\noindent If $u_i \in C^2(M_i \times \mathbb{R}^+)$ the vector function $(u_1,\cdots,u_n)$ is a regular solution to the Problem \eqref{P}-\eqref{initial}.
\end{Definition}

Now, let $(u_1,\cdots,u_n)$ be a regular solution to Problem \eqref{P}-\eqref{initial} and we define
\begin{equation} \label{def-uij}
u_{i,j}(x,t)=\chi_{j}(x,t)u_i(x,t),\ i=1,\cdots,n \ \mbox{and} \ j=1,2,
\end{equation}
where $\chi_{j}$ is a smooth cutoff function such that
\begin{equation} \label{def-cut}
\chi_{j}=\left\{ 
\begin{array}{lll}
1  \ \ \mbox{in} \ \  \overline{V_j} \times [0,T], \\
0  \ \ \mbox{in} \ \ \left( [0,L]\setminus \Omega_j\right)  \times [T+1,\infty),
\end{array}
\right.
\end{equation}
where $T$ is a positive constant satisfying \eqref{c}.

\noindent By definition, we can observe that $u_{i,j} \in C^2(M_{i,j} \times \mathbb{R}^+)$, 
where $M_{i,j}\equiv (\overline{\Omega}_j,g_i)$ represents a $1$-dimensional compact connex smooth riemannian manifold with boundary $\partial \Omega_j$ and with metric $g_i$.


Recall that the objective of the present section is to show a unique continuation property for the Problem \eqref{P}.  
For this purpose we shall assume 
\begin{equation} \label{w-hip}
u_i(x,t)=0, \ \ \forall (x,t) \in \omega \times [0,\infty), \ \ i=1,\cdots,n.
\end{equation}

Under the above notations, the function $u_{i,j}\in C^2(M_{i,j} \times \mathbb{R}^+)$ solves the problem
\begin{align} \label{Pij-0}
\left\{ 
\begin{array}{lll}
\begin{split}
\partial^2_{t} u_{i,j}-\Delta_i u_{i,j}&=f_i(u_{1,j},u_{2,j},\cdots,u_{n,j})  \,\,\, \mbox{in} \,\,\, M_{i,j}\times(0,T],\\
u_{i,j}(x,t)&=0  \,\,\, \mbox{on} \,\,\, \partial M_{i,j} \times (0,T],\\
u_{i,j}(x,0)&=\chi_{j}(x,0)u^0_i(x)\,\,\, \mbox{in} \,\,\, M_{i,j},\\
\partial_t u_{i,j}(x,0)&=\chi_{j}(x,0)u^1_i(x)  \,\,\, \mbox{in} \,\,\, M_{i,j}
\end{split}
\end{array}
\right.
\end{align}
In particular, the following decomposition is valid 
\begin{equation}\label{sum}
	u_i=u_{i,1}+u_{i,2}, \ \ (x,t) \in [0,L]\times [0,T] \ (i=1,\cdots,n).
\end{equation}

\noindent On the other hand, note that the system 
\begin{align} \label{v-0}
\left\{ 
\begin{array}{lll}
\begin{split}
\partial^2_{t} v_{i,j}-\Delta_i v_{i,j}&=f_i(v_{1,j},v_{2,j},\cdots,v_{n,j})  \,\,\, \mbox{in} \,\,\, M_{i,j}\times(0,T],\\
v_{i,j}(x,t)&=0  \,\,\, \mbox{on} \,\,\, \partial M_{i,j} \times (0,T],\\
v_{i,j}(x,0)&=\chi_{j}(x,0)u^0_i(x)\,\,\, \mbox{in} \,\,\, M_{i,j},\\
\partial_t v_{i,j}(x,0)&=\chi_{j}(x,0)u^1_i(x)  \,\,\, \mbox{in} \,\,\, M_{i,j}
\end{split}
\end{array}
\right.
\end{align}
is well posed, for all $i=1,\cdots,n$ and $j=1,2$.

\noindent Due to the uniqueness of solutions in the previous system (by density and continuity) we will focus on studying the problem \eqref{v-0}, in such a way that when assuming \eqref{w-hip} we have that $v_{i,j}=u_{i,j}$.

\subsection{Carleman estimates}

\noindent In the context of the previous section, in regard to the study the boundary terms for the system (\ref{v-0}), the following will be considered:

\begin{enumerate}
	\item For $v_{i,j} \in C(0,T;(H^1_0(M_{i,j})))\cap C^1(0,T;(L^2(M_{i,j})))$ the weak solution of Problem \eqref{v-0} we have 
	$$
	\<\nabla_{g_i} v_{i,j},\nabla_{g_i} d_j\PI_{g_i}=\<\nabla_{g_i} v_{i,j}, \nu\PI_{g_i}\<\nabla_{g_i} d_j, \nu\PI_{g_i},
	$$ 
	where $\nu$ denotes the outward unit normal field along the boundary $\partial M_{i,j}$.
	\item Let $\tau$ be a positive parameter then we define
	\begin{equation} \label{BTij}
	{BT}_{\tau}v_{i,j}\equiv 2\tau \int_{0}^{T}\int_{\partial \Omega_j}^{ }e^{2\tau \phi_j}
	\left(\<\nabla_{g_i} v_{i,j}, \nu\PI_{g_i}\right)^2\<\nabla_{g_i} d_j, \nu\PI_{g_i}\dx \dt,
	\end{equation}
	with $T>0$ satisfying \eqref{c}.
\end{enumerate}

Next, we shall investigate the properties on the forces $f_{i}(\cdot)$  $(i=1,\cdots,n)$ in the context of the new decomposition. 
We promptly have from  \eqref{f_i}-\eqref{Ei} that there exists a positive constant $C_T$ such that
\begin{equation} \label{esti-Fv}
\int_{0}^{T}\int_{\Omega_j}^{ }|f_i(v_{1,j},\cdots,v_{n,j})|^2_{g_i}\dx \dt\le C_T \sum_{i=1}^{n}\int_{0}^{T}\int_{0}^{L}\mathcal{V}_{i}(x,t)\dx \dt,
\end{equation}
where
\begin{equation} \label{V}
\mathcal{V}_{i}(x,t) \equiv |v_{i,1}+v_{i,2}|^2+ \gamma_{i}|\partial_x (v_{i,1}+v_{i,2})|^2+|\partial_t (v_{i,1}+v_{i,2})|^2,
\end{equation}
with
\begin{equation} \label{V22}
v_{i,1}+v_{i,2}=\left\{ 
\begin{array}{lll}
v_{i,1}  \ \ \mbox{in} \ \  [0,L_0] \times [0,T], \\
v_{i,2}  \ \ \mbox{in} \ \ [L_0,L] \times [0,T].
\end{array}
\right.
\end{equation}
Note that
$$
\mathcal{V}_i|_{\Omega_j\times [0,T]}=|v_{i,j}|^2+\gamma_{i}|\partial_x v_{i,j}|^2+|\partial_t v_{i,j}|^2 \ \ \mbox{and} \ \
\sum_{j=1}^{2}\int_{\Omega_j}^{}\mathcal{V}_i|_{\Omega_j\times [0,T]}\dx=\int_{0}^{L}\mathcal{V}_i \dx.
$$

Now, for any regular solution $v_{i,j}$, we find from \eqref{esti-Fv} that 
\begin{equation*}
\int_{0}^{T}\int_{\Omega_j}^{ }e^{2\tau\phi_j}|\partial^2_t v_{i,j}-\Delta_i v_{i,j}|_{g_i}^2\dx\dt\le C_{T}\sum_{i=1}^{n} \int_{0}^{T}\int_{0}^{L}e^{2\tau\phi_j}\mathcal{V}_i(x,t)\dx\dt.
\end{equation*}
In particular, by \eqref{sigma-min} and \eqref{Q}, we obtain 
\begin{equation*}
\int_{0}^{T}\int_{[Q(\sigma)]^c}^{ }e^{2\tau\phi_j}|\partial^2_t v_{i,j}-\Delta_i v_{i,j}|_{g_i}^2 \dx\dt\le C_{T}e^{2\tau \sigma}\sum_{i=1}^{n} \int_{[Q(\sigma)]^c}^{}\mathcal{V}_i(x,t)\dx\dt.
\end{equation*}

\begin{Remark} \label{k1,k2} Note that the Problem \eqref{v-0} satisfies the following compatibility condition
$$ 	v_{i,j}(L_0,t)=0  \,\,\, \mbox{in} \,\,\, (0,T] \ \ \mbox{and} \ \ 	v^0_i(L_0)=v^1_i(L_0)=0.$$ 
In particular, there exist positive constants $k_1, k_2$ such that
	$$ k_1\sum_{i=1}^{n}\int_{0}^{L}\mathcal{V}_{i}(x,t){\rm  d}x
	\le \sum_{i=1}^{n}\int_{0}^{L} \gamma_{i}|\partial_x (v_{i,1}+v_{i,2})|^2\!+\!|\partial_t (v_{i,1}+v_{i,2})|^2 {\rm  d}x
	\le k_2\sum_{i=1}^{n}\int_{0}^{L}\mathcal{V}_{i}(x,t){\rm  d}x.
	$$
\end{Remark}


Collecting all the above ingredients and proceeding analogously to Theorem 6.1 in \cite{Triggiani} we arrive at:  

\begin{Theorem}[Carleman Estimates]\label{carleman} 
Let $v_{i,j}$ $(i=1,\cdots,n$ and $j=1,2)$ be a regular solution of the Problem \eqref{v-0} 
with initial data $(\chi_{j}(x,0)u^0_i(x),\chi_{j}(x,0)u^1_i(x))$. 
Then, for all $\tau>0$ sufficiently large and $\varepsilon>0$ small, the following estimate holds true	
	\begin{align*}
	\sum_{j=1}^{2}\sum_{i=1}^{n}{BT}_{\tau}v_{i,j} &\ge \left[\frac{k_1e^{2\tau \sigma}(t_1-t_0)}{2}(\varepsilon \tau (1-c)-2nC_{T})e^{-C_{T} T} \right]\sum_{i=1}^{n}\ \int_{0}^{L}\Bigl[\mathcal{V}_i(x,0)+\mathcal{V}_i(x,T)\Bigl]{\rm  d}x  \\
	&\ \ \ \ -\left[\frac{C_{1,T}k_2e^{2\tau \sigma}}{2k_1}T e^{C_{T}T}+C_{T}\tau^3 e^{-2\tau \delta}\right]\sum_{i=1}^{n}\int_{0}^{L}\Bigl[\mathcal{V}_i(x,0)+\mathcal{V}_i(x,T)\Bigl]{\rm  d}x\\
	&\ge k_{T}  \sum_{i=1}^{n}\Bigl[\int_{0}^{L}\mathcal{V}_i(x,0)+\mathcal{V}_i(x,T)\Bigl]{\rm  d}x,
	\end{align*}
where
	\begin{itemize}
		\item [$1.$] $k_1, k_2$ the positive constants from Remark $\ref{k1,k2}$;
		\item [$2.$] $\sigma$ defined in $(\phi.2)$;
		\item [$3.$] $c\in (0,1)$ given in $(\ref{c})$;
		\item [$4.$] $C_{T},C_{1,T}$ positive constants depending only on $T,\sigma$ and $Q(\sigma)$;
		\item [$5.$] $k_{T}$ a positive constant depending only on $\sigma, \phi_1, \phi_2, n$ and $C_{1,T}$.
	\end{itemize}
	Moreover, the above inequality may be extended to all weak solution of the system \eqref{v-0} with initial data $(\chi_{j}(x,0)u^0_i(x),\chi_{j}(x,0)u^1_i(x)) \in H^1_0(M_{i,j})\times L^2(M_{i,j})$.
\end{Theorem}

\begin{Remark}
$(a)$ Note that the definition of $\sigma$ and $Q(\sigma)$ allow the right-hand term of previous inequality to be independent of $j$. $(b)$ The inequality holds for weak solution of the system \eqref{v-0} because {\rm\cite[Theorems 7.1 and 8.1]{Triggiani}}.
\end{Remark}

\subsection{A new UCP}

Thanks to the Theorem \ref{carleman} it is possible to state the main result of this part of paper.

\begin{Theorem}\label{mainUCP}
	Let $I$ be an open interval such that 
	$$
	\omega \equiv I \cap [0,L] \ne \emptyset .
	$$
	Then, for $T>0$ large enough, any weak solution 
	$$
	u_i \in C(0,T; H^1_0(0,L)) \cap C^1(0,T; L^2(0,L))
	$$ 
	of system \eqref{P}-\eqref{initial} vanishing in $\omega \times [0,T]$ must vanish all 
	over $[0,L] \times [0,T]$. 
\end{Theorem} 

\begin{proof} Without loss of generality we can assume 
	\begin{equation} \label{omega}
	\omega=\left( L_0-\frac{\varepsilon}{2},L_0+\frac{\varepsilon}{2}\right) \cap[0,L].
	\end{equation}
where $\varepsilon>0$ and $L_0 \in [0,L]$.

The proof of the unique continuation property will be divided into four steps:

{\it Step 1. Equivalence of systems.} 
Firstly, we observe that if $ (u_1,\cdots,u_n)$ is a weak solution of the Problem \eqref{P}-\eqref{initial} with overdetermined condition $(\ref{w-hip})$ 
then $u_{i,j}=\chi_j u_{i} \in C(0,T;H^1_0(M_{i,j})) \cap C^1(0,T;M_{i,j})$ is a solution of \eqref{Pij-0}, 
where $M_{i,1}=([0,L_0],g_i)$ and $M_{i,2}=([L_0,L],g_i)$, for all $i=1,\cdots,n$ and $j=1,2$.

From the Appendix A, for $j=1,2$, the system \eqref{v-0} with the following overdetermined condition
\begin{equation} \label{cond-v}
(v_{1,j},\cdots,v_{n,j})=(0,\cdots,0) \ \ \text{in $\omega_j \equiv \overline{\Omega}_j \cap \omega$}
\end{equation}
is well-posed and generates a strongly continuous semigroup 
$$
T_{M_{i,j}}:{H}_{\omega_j}(M_{i,j})\to {H}_{\omega_j}(M_{i,j})
$$
in the Hilbert space 
\begin{align*}
H_{\omega_j}(M_{i,j}) \equiv \left\{
 (v_{1,j},\cdots,v_{n,j},w_{1,j},\cdots,w_{n,j})  \left|
\begin{array}{l}
  v_{i,j} \in H^1_0 (M_{i,j}),\ w_{i,j} \in L^{2} (M_{i,j}), \\
  v_{i,j}=w_{i,j}=0 \  \text{in} \  \omega_j , i=1,\cdots,n
\end{array}
\right.
\right\}.
\end{align*}

\noindent In particular, 
$$
T_{M_{i,j}}(\chi_{j}(0)u^0_1,\cdots,\chi_{j}(0)u^0_n,\chi_{j}(0)u^1_1,\cdots,\chi_{j}(0)u^1_n)=(u_{1,j},\cdots,u_{n,j},\partial_t u_{1,j},\cdots,\partial_t u_{n,j}).
$$
is the weak solution of \eqref{v-0} satisfying \eqref{cond-v} with initial data
$$
(\chi_{j}(0)u^0_1,\cdots,\chi_{j}(0)u^0_n,\chi_{j}(0)u^1_1,\cdots,\chi_{j}(0)u^1_n) \in {H}_{\omega_j}(M_{i,j}).
$$

\medskip
{\it Step 2. Carleman estimate.} 
From Step 1 and via Theorem \ref{carleman} there exists a positive constant $k_T$ such that
$$
\sum_{j=1}^{2}\sum_{i=1}^{n}{BT}_{\tau}u_{i,j} \ge  k_{T}  \sum_{i=1}^{n}\left[ \int_{0}^{L}\mathcal{V}_i(x,0)+\mathcal{V}_i(x,T)\dx\right].
$$
Next, from \eqref{E}, \eqref{sum}, \eqref{BTij}, \eqref{V} and \eqref{V22} we find that 
\begin{align}\label{step2}
2\tau \sum_{j=1}^{2}\sum_{i=1}^{n}\int_{0}^{T}\int_{\partial \Omega_j}^{ }e^{2\tau \phi_j}
\left(\<\nabla_{g_i} u_{i,j}, \nu\PI_{g_i}\right)^2\<\nabla_{g_i} d_j, \nu\PI_{g_i}\dx \dt\ge k_T\left(F_u(0)+F_u(T)\right).
\end{align}

\medskip 

{\it Step 3. Boundary estimates.} The fact that $L_0 \in \omega$ and $\omega$ is an open subset of $[0,L]$ we have
$$
\<\nabla_{g_i} u_{i,j}(L_0), \nu(L_0)\PI_{g_i}=0,   \ \forall i=1,\cdots,n  \ \text{and}   \ j=1,2.
$$

\noindent Now, recalling Item 5 from Lemma \ref{lemaover} we obtain 
$$
\<\nabla_{g_i} d_j, \nu\PI_{g_i}<0 \ \ \text{in $\{0,L\}$},   \ \forall i=1,\cdots,n  \ \text{and}   \ j=1,2,
$$

\noindent Combining the above information with assumption that $\tau>0$, we infer that
\begin{align}\label{step3}
2\tau \sum_{j=1}^{2}\sum_{i=1}^{n}\int_{0}^{T}\int_{\partial \Omega_j}^{ }e^{2\tau \phi_j}
\left(\<\nabla_{g_i} u_{i,j}, \nu\PI_{g_i}\right)^2\<\nabla_{g_i} d_j, \nu\PI_{g_i} \dx \dt \leqslant 0.
\end{align}

\medskip 

{\it Step 4. Conclusion.} From inequalities $(\ref{step2})$ and $(\ref{step3})$ we find  
$$
0\ge k_T\left(  F_u(0)+F_u(T)\right) \ge0.
$$
This last implies that $F_u(0)=0$. Since \eqref{P}-\eqref{initial} is well-posed, the result is followed.

\medskip


\end{proof}

\setcounter{equation}{0}
\section{Dynamics of locally damped Bresse systems}
Let  us consider the semilinear Bresse system
\begin{align} \label{P01}
\left\{ 
\begin{array}{rrr}
\begin{split}
\rho_{1}\varphi_{tt} - k(\varphi_{x}+\psi+l  w)_{x} - k_{0}l  (w_{x}-l \varphi)+a_{1}(x)g_{1}(\varphi_t) + f_1(\varphi,\psi,w) =0,    \\
\rho_{2}\psi_{tt} - b\psi_{xx}+k(\varphi_{x} + \psi+l  w) + a_{2}(x) g_{2}(\psi_t)+ f_2(\varphi,\psi,w) = 0,  \\
\rho_{1}w_{tt} - k_{0}(w_{x}-l \varphi)_{x} + kl (\varphi_{x}+\psi+l  w) +a_{3}(x) g_{3}(w_t) + f_3(\varphi,\psi,w) =0,
\end{split}
\end{array}
\right.
\end{align}
with Dirichlet boundary conditions
\begin{equation} \label{bci}
\varphi(0,t) = \varphi(L,t) = \psi(0,t) = \psi(L,t) = w(0,t) = w(L,t) = 0, \quad t \in \mathbb{R}^{+},
\end{equation}
and with initial condition
\begin{equation} \label{3ci}
\varphi(0) = \varphi_{0},\:\: \varphi_t(0) = \varphi_{1},\:\: \psi(0) = \psi_{0},\:\: \psi_t(0) = \psi_{1}, \:\: w(0)= w_0, \:\: w_t(0) = w_1.
\end{equation}

\subsection{Well-posedness}

In this section, we summarize all the assumptions that will be used to prove the main result. 
We also introduce the well-posedness result along with some energy inequalities.  

\medskip 
\noindent \textbf{Notations.} Henceforth the symbols $L^p(0,L)$ $(p\geqslant 1)$ and $H^m(0,L)$ $(m \in \mathbb{N})$ denote 
the Lebesgue and Sobolev spaces, respectively. 
The norms in $L^p (0,L)$ are indicated by $\|\cdot\|_{p}$ and $\|\cdot\|_{L^{2}(0,L)}\equiv \|\cdot\|$. We will also frequently use the inequality$$
\Vert u \Vert \ls \frac{L}{\pi} \Vert u_x \Vert,\,\, \forall\,u \in H^{1}_{0}(0,L).
$$
\medskip 

\noindent \textbf{Assumptions.} The following hypotheses will be used throughout the paper. 

\noindent $(f.1.)$ The sources functions $f_i \!\in\! C^{1}(\mathbb{R})$ $(i=1,2,3)$ are locally Lipschitz and there exists a function 
$F \!\in\! C^1(\mathbb{R}^3)$ such that $\nabla F = (f_1,f_2,f_3)$.

\noindent $(f.2.)$ There exists constants $0 \leq \alpha < \frac{\pi^2}{2 \beta L^2}$ and  $c_F>0$ such that
\begin{align*}
F(u,v,w)  &\geq  - \alpha \Big[ |u|^{2} + |v|^{2} + |w|^2 \Big] - c_F, \,\, \forall \, u,v,w \in \mathbb{R},\\
\nabla F(u,v,w) \! \cdot \! (u,v,w)  &\geq F(u,v,w) - \alpha \Big[ |u|^{2} + |v|^{2} + |w|^{2} \Big] - c_F, \,\, \forall \, u,v,w \in \mathbb{R},
\end{align*}
where  $\beta>0$ is the constant
\begin{equation*}
\|\varphi_x \|^2 + \|\psi_x\|^2 + \|w_x \|^2 \ls  \beta \Big[ b\|\psi_x\|^2 + k \|\varphi_x + \psi +l w\|^2+ k_0 \|w_x - l \varphi \|^2 \Big].
\end{equation*}

\noindent $(f.3.)$ There exists $c_f > 0$ such that
\begin{equation*}
|\nabla f_i(u,v,w) | \ls c_f \Big[ 1 + |u|^{p-1} + |v|^{p-1} + |w|^{p-1} \Big],\, i=1,2,3,\, p\geqslant 1, \, \forall \, u,v,w \in \mathbb{R}.
\end{equation*}

\noindent $(g.1.)$ The damping  functions $g_{i}\!\in\! C^1(\mathbb{R})$ $(i \!=\! 1,2,3)$ are monotone increasing with $g_{i}(0)\!=\! 0$. 
Moreover, we assume that there exist constants positive constants $m$ and $M>0$ such that
\begin{equation*}
m\leq g^{\prime}_{i}(s) \leq M, \,\,\forall \, s \in \mathbb{R}.
\end{equation*} 

\noindent $(a.1.)$ The localizing functions $a_{i} \!\in\! L^\infty(0,L)$ $(i \!=\! 1,2,3)$ 
are non-negative and there exists  positive constant $a_{0}$ such that
$$a_{i}(x)\geqslant a_{0},\,\, x \in I_{i},\,\, i=1,2,3,$$
where $I_{i}\subset[0,L]$ are open intervals  with $(L_{1},L_{2}) \equiv \bigcap_{i} I_{i}\neq \emptyset$.
\medskip 

\noindent \textbf{Dynamical system generation.} 
Before introducing the well-posedness result, we start with the necessary functional framework.  
First, the finite energy space $H$ of the well-posedness is defined as 
\begin{equation*}
H = H^1_0 (0,L) \times H^1_0 (0,L) \times H^1_0 (0,L) \times  L^2 (0,L)\times  L^2 (0,L)\times  L^2 (0,L).
\end{equation*} 
For $\Z=(\varphi(t), \psi(t), w(t),\tilde{\varphi}(t),\tilde{\psi}, \tilde{w}(t))\!\in\!H$, we define the $H$ norm as 
\begin{equation} \label{norm}
\|\Z\|_{H}^2 =
\rho_1\|\tilde{\varphi}\|^2 +\rho_2\|\tilde{\psi}\|^2+\rho_1\|\tilde{w}\|^2 + b\|\psi_x\|^2 + k \|\varphi_x + \psi +l  w\|^2+ k_0 \|w_x - l  \varphi \|^2.
\end{equation}

Next, let $A : D(A) \subset H \to H$ be the differential operator 
$$
 A\Z 
    \equiv
  \left[ \begin{array}{c}
   \tilde{\varphi} \\
	   \tilde{\psi} \\
	   \tilde{w} \\
      {k}{\rho_{1}}^{-1}  (\varphi_{x}+\psi+l  w)_{x} + {k_{0}l }{\rho_{1}}^{-1}(w_{x}-l \varphi)\\
      {b}{\rho_{2}}^{-1} \psi_{xx}-{k}{\rho_{2}}^{-1}(\varphi_{x} + \psi+l  w) \\
      {k_{0}}{\rho_{1}}^{-1}(w_{x}-l \varphi)_{x} - kl \rho_{1}^{-1}(\varphi_{x}+\psi+l  w) 
    \end{array}\right] ,
$$
with domain
$$
D(A)=\bigl[ H^{2}(0,L)\cap H^{1}_{0}(0,L)\bigl]^{3} \times H^{1}_{0}(0,L)^{3}.
$$
Next, let be $B:D(B)=H \to H$ the damping operator 
$$
B \Z\equiv
  \left[ \begin{array}{c}
		0\\
		0\\
		0\\
      - \alpha_{1}(x)\rho_1^{-1} g_{1}(\tilde{\varphi})  \\
      - \alpha_{2}(x)\rho_2^{-1} g_{2}(\tilde{\psi})\\
      - \alpha_{3}(x)\rho_1^{-1} g_{3}(\tilde{w}) 
  \end{array}\right].
$$
Finally, by $\mathscr{F}: H \to H$, we represent the source terms operator
$$
\mathscr{F}\Z\equiv
  \left[ \begin{array}{c}
  0\\ 
  0\\
  0\\ 
      -\rho_1^{-1}f_1(\varphi,\psi,w)     \\
      -\rho_2^{-1} f_2 (\varphi,\psi,w)   \\
      -\rho_1^{-1} f_3(\varphi,\psi,w)
 \end{array}\right].
$$

Now, using the definitions of operators $A, B, \mathscr{F}$, we can abstract represent the problem as follows
\begin{equation} \label{PC}
\frac{\mbox{d}}{ \dt}\Z(t)\!-\!
\left(A + B \right)\Z(t) = \mathscr{F}(\Z(t)), 
\,\,\Z(0)\equiv\Z_{0}=(\varphi_{0}, \psi_{0}, w_{0}, \varphi_{1},\psi_{1},w_{1}),
\end{equation}
where
$$
\Z(t) =(\varphi(t), \psi(t), w(t),\tilde{\varphi}(t),\tilde{\psi}, \tilde{w}(t))\,\,\mbox{with} 
\,\, \tilde{\varphi} = \varphi_{t}, \; \tilde{\psi}=\psi_{t},\; \tilde{w}=w_t.
$$

\medskip
We observe that the well-posedness of (\ref{PC}) induces the well-posedness for the Problem (\ref{P01}) -(\ref{3ci}). 
In the following, we present the well-posedness for (\ref{PC}). 

\begin{Theorem} [Well-posedness] \label{theo-existence} Assume the validness of Assumptions $(f.1)$-$(g.1)$.
Then for any initial data ${\normalfont \textbf{Z}}_{0} \in H$ and $T>0$, the Cauchy problem $(\ref{PC})$ admits a unique weak solution
${\normalfont \textbf{Z}} \in C([0,T];H)$ that depends continuously on the initial data and is given by the variation of parameters formula
\begin{equation} \label{SolFormula}
{\normalfont \textbf{Z}}(t)= e^{(A+B)t}{\normalfont \textbf{Z}}_{0}
+\int_{0}^{t}e^{(A+B)(t-s)}\mathscr{F}({\normalfont \textbf{Z}}(s)) {\rm d} s, \;\; t \in [0,T]. 
\end{equation}
Moreover, if ${\normalfont \textbf{Z}}_0 \in D(A)$ then the solution is strong.
\end{Theorem}

The well-posedness result stated above is known and can be found in \cite{MaMonteiro}. 
Theorem $\ref{theo-existence}$, in particular, implies that the map 
$ H \ni \Z_{0} \mapsto \Z(t)=(\varphi, \psi, w,\varphi_{t},\psi_{t}, w_{t})$, where $\Z(t)$ solves (\ref{PC}),  
defines a strongly continuous semigroup $\{S(t)\}_{t \geqslant 0}$ on $H$.

\medskip 
\noindent \textbf{Energy.} 
Let $\Z(t)=(\varphi(t),\psi(t),w(t),\varphi_{t}(t),\psi_{t}(t),w_{t}(t))$ be a solution of $(\ref{P01})$-$(\ref{3ci})$. 
The energy is defined by the following functional
\begin{equation} \label{energy-def}
\mathscr{E}_{\Z}(t) \equiv {E}_{\Z}(t) + \int_{0}^{L}\! F(\varphi,\psi,w)\dx =  \|\Z(t)\|_{H}^{2} + \int_{0}^{L}\! F(\varphi,\psi,w)\dx.
\end{equation}

\noindent The weak solution $\Z(t)=(\varphi(t),\psi(t),w(t),\varphi_{t}(t),\psi_{t}(t),w_{t}(t))$ satisfies the energy identity      
\begin{equation} \label{energy-ID}
 \mathscr{E}_{\Z}(t) + \int_{s}^{t}\int_{0}^{L} \Bigl[\alpha_{1}(x)g_{1}(\varphi_t)\varphi_{t}
+\alpha_{2}(x)g_{2}(\psi_t)\psi_{t}
+\alpha_{3}(x)g_{3}(w_t)w_{t}\Bigl]
\dx{\rm d}\tau =  \mathscr{E}_{\Z}(s),
\end{equation}
for all $0 \leqslant s < t$.

\noindent As in \cite{MaMonteiro}, the energy $\mathscr{E}_{\Z}(\cdot)$ (\ref{energy-def}) and the norm $\|\cdot\|_{H}$ (\ref{norm}) satisfy 
\begin{equation} \label{energia-domina}
 C_{E}\|\Z(t)\|_{H}^{2} - Lc_F \leqslant \mathscr{E}_{\Z}(t) 
\leqslant  \|\Z(t)\|_{H}^{2} + c_{E}(1+\|\Z(t)\|_{H}^{p+1}),\,\,\forall \, t \geqslant  0,
\end{equation}
for some  positive constants $C_{E}$ and $c_{E}$.

\subsection{Global attractors} 

\subsubsection{Abstract existence theorems} 

Some essential definitions and results from the theory of attractors for gradient systems is collected 
\begin{Definition}
A global attractor for a dynamical system $(H,S(t))$, with evolution operator $\{S(t)\}_{t \geqslant 0}$ 
on a complete metric space $H$ is defined as a  a compact set $\mathscr{A} \!\subset\! H$ that is fully
invariant, that is $S(t)\mathscr{A} = \mathscr{A}$  for all $t \geqslant 0$, and uniformly attracts all bounded subsets of $H$
\begin{align*}
\lim_{t\to \infty} \sup\Big\{{\rm dist}_{H}(S(t)\normalfont{\Z},\mathscr{A})\,|\, \normalfont{\Z}\in \mathscr{B}\Big\} 
= 0,\,\mbox {for any bounded set} \,\, \mathscr{B}\!\subset\! H.
\end{align*}
\end{Definition}

\begin{Definition}
The fractal dimension of a compact set $\mathscr{A} \subset H$ in a metric space $H$ is defined as
$$
\dim_{f}^{H}\!(\mathscr{A}) = \limsup_{\varepsilon \to 0} \frac{\ln N_{\varepsilon}(\mathscr{A})}{\ln(1/{\varepsilon})} ,
$$
where $N_{\varepsilon}(\mathscr{A})$ is the minimal number of closed balls of radius $\varepsilon$ which
cover the set $A$.
\end{Definition}

To ascertain the existence of a global attractor, we use the concept of gradient and quasi-stable dynamical systems. 
The global attractor for this systems admits additional structure and properties: 
(i) the attractor for gradient systems has a regular structure, that is, 
the attractor is described by the unstable manifold emanating from the set of stationary points and 
(ii) quasi-stable systems provide several properties of attractors, such as finite dimensionality. 

\begin{Definition}\label{Lya} Let $\mathscr{Y}\subset H$ be a forward invariant set of a dynamical
system $(H,S(t))$. 
${\rm(i)}$ A continuous functional $\Phi: \mathscr{Y} \to \mathbb{R}$ is said to be a Lyapunov
function on $\mathscr{Y}$ for $(H,S(t))$ if the map $t \mapsto \Phi(S(t)\normalfont{\Z})$ is non-increasing for any $\normalfont{\Z} \in \mathscr{Y}$. 
${\rm(ii)}$ The Lyapunov function is said to be strict on $\mathscr{Y}$ if the equation if $\Phi(S(t)\normalfont{\Z})=\Phi(\normalfont{\Z})$ for all $t>0$ 
for some $\normalfont{\Z} \in \mathscr{Y} $ implies that $y$ is a stationary point of $(H,S(t))$. 
${\rm(iii)}$ The dynamical system  $(H,S(t))$ is said to be gradient if there exists a strict Lyapunov function on the whole phase space $H$.
\end{Definition}

\begin{Definition}\label{qs}
Let $X,Y$ be reflexive Banach spaces, $X$ compactly embedded in $Y$. 
We consider a dynamical system $(H,S(t))$ with $H= X \times Y$ and evolution operator defined by
\begin{equation} \label{DS}
S(t)\normalfont{\Z}= (u(t),u_{t}(t)), \qquad \normalfont{\Z}=(u(0), u_{t}(0)) \in {H},
\end{equation}
where the function $u$ possess the property
\begin{equation}\label{Regularity}
u \in C([0,\infty); X)\cap C^{1}([0,\infty); Y).
\end{equation}
A dynamical system of the the form $(\ref{DS})$ with regularity $(\ref{Regularity})$
is said to be quasi-stable on a set
$\mathscr{B} \subset H$, if there exist a compact semi-norm $[ \, \cdot \, ]_{X}$
on $X$ and non-negative scalar functions $a(t), b(t), c(t)$, such that, 
${\rm(i)}$ $a(t), b(t)$ are locally bounded on $[0,\infty)$,
${\rm(ii)}$ $b(t)\in L^1(0,\infty)$ with $\lim_{t\to \infty}b(t)=0$ and 
${\rm(iii)}$ for any $\normalfont{\Z}^1,\normalfont{\Z}^2 \in \mathscr{B}$ the following estimates hold true
\begin{equation} \label{L}
\| S(t)\normalfont{\Z}^1 - S(t)\normalfont{\Z}^2 \|_{H}^2 \ls a(t)\|\normalfont{\Z}^1 - \normalfont{\Z}^2 \|_{H}^2 ,
\end{equation}
and
\begin{equation} \label{SI}
\|S(t)\normalfont{\Z}^1 - S(t)\normalfont{\Z}^2 \|_{H}^2 \ls  b(t) \| \normalfont{\Z}^1 -\normalfont{\Z}^2 \|_{H}^2 
+ c(t) \sup_{0<s<t} [ u^1(s)-u^2(s) ]_{X}^2 ,
\end{equation}
where $S(t)\normalfont{\Z}^{i}= (u^{i}(t),u_{t}^{i}(t))$, $i=1,2$.
\end{Definition}

Unifying the abstracts results from \cite{Yellow} 
we arrive at the following criteria for existence and properties of global attractors.

\begin{Theorem} \label{Conclusion}
Let $(H,S(t))$ be a gradient quasi-stable dynamical system. 
Assume its Lyapunov function $\Phi(\cdot)$ is bounded from above on any bounded
subset of $H$ and the set $\Phi(R) = \big\{ \normalfont{\Z}\in H\,|\, \Phi(\Z)\ls R \big\}$ is bounded for every $R$. 
If the set $\mathscr{N}$ of stationary points of $(H,S(t))$  is bounded, then $(H,S(t))$ possesses a finite dimensional global
attractor $\mathscr{A}$ defined by the unstable manifold emanating from set of stationary solution. 
Moreover, any trajectory stabilizes to the set $\mathscr{N}$ of stationary points, that is, 
$$\lim_{t\to +\infty}{\rm dist}_{H}(S(t)\normalfont{\Z},\mathscr{N})=0,\,\, \forall\,  \normalfont{\Z}\in H.$$
\end{Theorem}

\medskip 
We now state the main result of the present chapter.
\begin{Theorem} \label{Main}
Under the Assumptions $(f.1)$-$(a.1)$ the dynamical system $(H,S(t))$ generated by the problem
$(\ref{P01})$-$(\ref{3ci})$ has a global attractor $\mathbf{A}$ characterized by
$$
\mathbf{A} =\mathbb{M}_{+}(\mathcal{N}),
$$
where $\mathbb{M}_{+}(\mathcal{N})$ is the unstable manifold emanating from $\mathcal{N}$, 
the set of stationary points of $\{S(t)\}_{t \ge 0}$.
\end{Theorem}

\subsubsection{Gradient structure and quasi-stability}

Our strategy centers on establishing the conditions from the Theorem \ref{Conclusion}.  
Starting exhibiting the gradient structure for $(H, S(t))$ and focusing our attention on the strictness of the Lyapunov function 
where the new observability result stated in Theorem \ref{mainUCP} plays an essential role in the proof.

\begin{Proposition} 
\label{P-Gradient} 
Let the assumptions of Theorem $\ref{Main}$ be satisfied. 
Then, $(H,S(t))$ is a gradient dynamical system.
\end{Proposition}
\proof The dynamical system $(H,S(t))$ is gradient with full energy $\mathscr{E}_{\Z}(\cdot)$ 
- defined in (\ref{energy-def}) - being the strict Lyapunov function $\Phi(\cdot)$. 
In fact, from identity (\ref{energy-ID}), we find that  $t \rightarrow \Phi(S(t)\Z)$ 
is a non-increasing function for any $\Z \in H$.
 
Next, we suppose that $\Phi(S(t)\Z)=\Phi(\Z)$, for all $t > 0$. Then, from identity (\ref{energy-ID}), we obtain 
\begin{equation*}
\int_{s}^{t}\int_{0}^{L} \Bigl[\alpha_{1}(x)g_{1}(\varphi_t)\varphi_{t}
+\alpha_{2}(x)g_{2}(\psi_t)\psi_{t}
+\alpha_{3}(x)g_{3}(w_t)w_{t}\Bigl]
\dx{\rm d}\tau =  0.
\end{equation*}

\noindent This shows that $\varphi_{t}\!=\! \psi_{t}\!=\! w_{t}\!=\!0$ a.e. in $(L_{1},L_{2})\times [0,T] $, where $(L_{1},L_{2})=\bigcap_{i} I_{i}$. 
Thus, $(\varphi,\psi,w)$ satisfies the problem 

\begin{align} \label{uc1}
\left\{ 
\begin{array}{rrr}
\begin{split}
\rho_{1}\varphi_{tt} - k(\varphi_{x}+\psi+l  w)_{x} - k_{0}l  (w_{x}-l \varphi) + f_1(\varphi,\psi,w) =0  & \;\;  \mbox{in} \:\: (0,L) \times [0,T], \\
\rho_{2}\psi_{tt} - b\psi_{xx}+k(\varphi_{x} + \psi+l  w) + f_2(\varphi,\psi,w) = 0 &\;\;  \mbox{in} \:\: (0,L) \times [0,T],  \\
\rho_{1}w_{tt} - k_{0}(w_{x}-l \varphi)_{x} + kl (\varphi_{x}+\psi+l  w) + f_3(\varphi,\psi,w) =0  &\;\;  \mbox{in} \:\: (0,L) \times [0,T].
\end{split}
\end{array}
\right.
\end{align}

\noindent Using the notation $u^{1}\!=\!\varphi_{t}$, $u^{2}\!=\!\psi_{t}$, $u^{3}\!=\!w_{t}$ 
and taking the derivative in the distributional sense of $(\ref{uc1})$, we find that
$(u^{1},u^{2},u^{3})$ solves the problem 
\begin{align*}
\left\{ 
\begin{array}{rrr}
\begin{split}
u^{1}_{tt} - \gamma_{1} u^{1}_{xx} = F^1(u^{1},u^{2},u^{3})    & \;\;  \mbox{in} \:\: (0,L) \times [0,T],  \\
u^{2}_{tt} - \gamma_{2} u^{2}_{xx} = F^2(u^{1},u^{2},u^{3})    & \;\;  \mbox{in} \:\: (0,L) \times [0,T],  \\
u^{3}_{tt} - \gamma_{2} u^{3}_{xx} = F^3(u^{1},u^{2},u^{3})    & \;\;  \mbox{in} \:\: (0,L) \times [0,T],
\end{split}
\end{array}
\right.
\end{align*}
with $\gamma_{1}=\frac{k}{\rho_{1}}$, $\gamma_{2}=\frac{b}{\rho_{2}}$, $\gamma_{3}=\frac{k_{0}}{\rho_{1}}$ and with forcing $F^{i}(\cdot)$ defined by
\begin{align*}
F^1(u^{1},u^{2},u^{3}) &= \gamma_{1} u^{2}_{x} + \gamma_{1} u^{3}_{x} + k_{0}\rho_{1}^{-1}l  (u^{3}_{x}-l  u^{1}) - \rho_{1}^{-1} \partial_{t}[f_{1}(\varphi,\psi,w)], \\
F^2(u^{1},u^{2},u^{3}) &= - k\rho_{2}^{-1}(u^{1}_{x} - u^{2}+l  u^{3}) - \rho_{2}^{-1} \partial_{t}[f_{2}(\varphi,\psi,w)],  \\
F^3(u^{1},u^{2},u^{3}) &= - \gamma_{3}l  u^{1}_{x} - k\rho_{1}^{-1}l (u^{1}_{x}+u^{2}+l  u^{3}) - \rho_{1}^{-1} \partial_{t}[f_{3}(\varphi,\psi,w)]. 
\end{align*}
Now, we apply the UCP - Theorem \ref{mainUCP} - to conclude that $(\varphi_{t},\psi_{t},w_{t})\!=\!(u^{1},u^{2},u^{3})\!=\!(0,0,0)$. 
Therefore, the solution $\Z\! \in H$ must be stationary. This implies that the energy $\mathscr{E}_{\Z}(\cdot)$ is strict on $H$.  
\qed

\medskip
 
Our next aim is to show the quasi-stability of $(H,S(t))$. According to the Definition \ref{qs}, 
the difference of two trajectories should obeys estimates $(\ref{L})$ and  $(\ref{SI})$. 
Taking the advantage of the locally Lipschitz property of $f_{i}(\cdot)$ and the variation of parameter formula (\ref{SolFormula}), 
one can easily show the validity of $(\ref{L})$. 
Next, by means of multiplier technique, we prove the stabilization inequality $(\ref{SI})$.  

\begin{Proposition} 
\label{P-Quasi} 
Let the assumptions of Theorem $\ref{Main}$ be satisfied. 
Then, $(H,S(t))$ is a quasi-stable dynamical system.
\end{Proposition}
\proof The proof is carried out through several energy estimates. In the text that follows, we use the notations
$$
\tilde{v}=v^{1}-v^{2},\,\, G(v)=g_{i}(v^{1})-g_{i}(v^{2})\,\, \mbox{and}\,\, F_i(v)=f_i(v^{1})-f_i(v^{2}),\,\, i= 1,2,3.
$$
 
First, we shall consider the difference of two trajectories with initial data $\Z^{1}_{0},\Z^{2}_{0} \in \mathscr{B}$, 
where $\mathscr{B}$ is a bounded subset of $H$. 
The corresponding solution $S(t)(\Z^{1}_{0}-\Z^{2}_{0})\equiv\Z^{1}-\Z^{2}=(\tphi, \tpsi, \tw,\tphi_{t}, \tpsi_{t}, \tw_{t})$ 
verifies the following problem
\begin{align} \label{B4}
\left\{ 
\begin{array}{rrr}
\begin{split}
\rho_{1}\tphi_{tt} - k(\tphi_{x}+\tpsi+l  \tw)_{x} - k_{0} l  (\tw_{x}-l \tphi) & = - a_{1}(x)G_1(\tphi_t) - F_1(\tphi,\tpsi,\tw), \\
\rho_{2}\tpsi_{tt} - b\tpsi_{xx}+k(\tphi_{x} + \tpsi+l  \tw)
& = - a_{2}(x)G_2(\tpsi_t) - F_2(\tphi,\tpsi,\tw),  \\
\rho_{1}\tw_{tt} - k_{0}(\tw_{x}-l \tphi)_{x} + kl  (\tphi_{x}+\tpsi+l  \tw)
& = - a_{3}(x)G_3(\tw_t) - F_3(\tphi,\tpsi,\tw),   
\end{split}
\end{array}
\right.
\end{align}
with zero Dirichlet boundary conditions and initial conditions $\Z^1_{0}-\Z^2_{0}$.

Second, we let $\epsilon_{0}$  be a positive real number, such that, $\epsilon_{0} \ls (L_{2}-L_{1})$, 
where $(L_{1},L_{2}) = \bigcap_{i} I_{i}$. 
We consider the following real-function $\xi(\cdot)$ defined as follows
 \begin{align*}              
&\xi(x)=\left \{
\begin{array}{l}
\displaystyle{(\lambda -1 )x,\,\, \mbox{if}\,\, x \in [0,L_{1}+\epsilon_{0})} ,\\
\displaystyle{\lambda(x-L_{1}-\epsilon_{0})+(L_{1}-L_{2}+2\epsilon_{0})/(L_{1}+\epsilon_{0}),\,\, \mbox{if}\,\, x \in (L_{1}+\epsilon_{0},L_{2}-\epsilon_{0}]} ,\\
\displaystyle{(\lambda-1)(x-L),\,\, \mbox{if}\,\, x \in [L_{2}-\epsilon_{0},L]}.		                   														                   
\end{array}
\right. .
\end{align*}

\noindent Now, we take $\tphi_{x}\xi,\tpsi_{x}\xi,\tw_{x}\xi$ as multipliers for $(\ref{B4})$. Thus, we find
\begin{align}\label{est1a}
\begin{split}
\frac{1}{2}\int_{0}^{T}\int_{0}^{L}&(\lambda-1)E(t)\,\dx\dt \\
&=-\int_{0}^{L}\Big[\tphi_{t}\tphi_{x}+\tpsi_{t}\tpsi_{x} +\tw_{t}\tw_{x}\Big]\xi\,\dx\Bigl|_{0}^{T}
-\frac{1}{2}\int_{0}^{T}\int_{L_{1}+\epsilon_{0}}^{L_{2}-\epsilon_{0}}E(t)\,\dx\dt\\
&\quad - \int_{0}^{T}\int_{0}^{L}\Big[ k(\tphi_{x}-\tpsi+l  \tw)(\tpsi+l \tw) + k_{0}l (\tw_{x}-l \tphi)\tphi\Big]\xi^{\prime}\,\dx\dt \\
&\quad - \int_{0}^{T}\int_{0}^{L}\Big[a_{1}G_{1}(\tphi_t)\tphi_{x} + a_{2}G_{2}(\tpsi_t)\tpsi_{x}+ a_{3}G_{3}(\tw_t)\tw_{x}\Big]\xi\,\dx\dt   \\
&\quad - \int_{0}^{T}\int_{0}^{L}\Big[ F_1(\tphi,\tpsi,\tw) \tphi_{x} + F_2(\tphi,\tpsi,\tw) \tpsi_{x} + F_3(\tphi,\tpsi,\tw) \tw_{x}\Big]\xi\,\dx\dt, 
\end{split}
\end{align} 
where
$$E(t)= \rho_1|\tilde{\varphi}_{t}|^2 +\rho_2|\tilde{\psi}_{t}|^2+\rho_1|\tilde{w}_{t}|^2 + b|\tpsi_x|^2 + k |\tphi_x + \tpsi +l  \tw|^2+ k_0 |\tw_x - l  \tphi |^2.$$

\noindent Let us estimate the left-hand side of $(\ref{est1a})$. 
Note that, from the definition of energy, we find ${E_{\Z}(t)\!=\!\int_{0}^{L} E(t)\dt\!=\!\|\Z\|^{2}_{H}}$. 
Then, we can show that there exists $c>0$ satisfying 
$$\left|\int_{0}^{L}\Big[\tphi_{t}\tphi_{x}+\tpsi_{t}\tpsi_{x} +\tw_{t}\tw_{x}\Big]\, \xi \dx\right| \ls c \sup_{x\in [0,L]}\{\xi(x)\} {E}_{\Z}(t).$$

\noindent This last implies that
\begin{align}\label{es11}\left|\int_{0}^{L}\Big[\tphi_{t}\tphi_{x}+\tpsi_{t}\tpsi_{x} +\tw_{t}\tw_{x}\Big]\, \xi \dx\Bigl|_{0}^{T}\right|\ls c \sup_{x\in [0,L]}\{\xi(x)\}
({E}_{\Z}(T)+{E}_{\Z}(0)).
\end{align}

\noindent Also using the definition of $E_{\Z}(t)$, one obtains
\begin{align}\label{es22}
\begin{split}
\left|\int_{0}^{T}\int_{0}^{L}\Big[ k(\tphi_{x}-\tpsi+l  \tw)(\tpsi+l \tw) + k_{0}l (\tw_{x}-l \tphi)\tphi\Big]\xi^{\prime}\,\dx\dt \right|\\
\ls \epsilon \int_{0}^{T}\!\! E_{\Z}(t)\dt + c_{\epsilon}\mbox{l.o.t}(\tphi,\tpsi,\tw),
\end{split}
\end{align}
\noindent with lower order terms defined by
$$\mbox{l.o.t}(\tphi,\tpsi,\tw)\equiv \sup_{\sigma \in [0,T]}\Bigl[ \| \tphi(\sigma) \|_{2p}^2
+ \| \tpsi (\sigma)\|_{2p}^2 + \| \tw (\sigma)\|_{2p}^2 \Bigl].$$

\noindent To estimate the damping terms, we use Assumption $(g.1)$ to obtain 
\begin{align*}\begin{split}
\left|\int_{0}^{T}\int_{0}^{L}a_{1}G_1(\tphi_{t})\tphi_{x}\xi\,\dx\dt\right|
&\ls \sup_{x\in [0,L]}\{\xi(x)\}\int_{0}^{T}\int_{0}^{L}a_{1}M|\varphi_{t}^{1}-\varphi_{t}^{2}||\,\tphi_{x}|\,\dx\dt
\\
&\ls \frac{\epsilon}{3} \int_{0}^{T}\!\! E_{\Z}(t)\dt + c_{\epsilon} \int_{0}^{T}\int_{0}^{L}a_{1}\tphi_{t}^{2}\,\dx\dt.
\end{split}
\end{align*}

\noindent This allows us to conclude the following estimate
\begin{align} \label{Gest1}\begin{split}
\left|\int_{0}^{T}\int_{0}^{L}
\Big[a_{1}G_{1}(\tphi_t)\tphi_{x} + a_{2}G_{2}(\tpsi_t)\tpsi_{x}+ a_{3}G_{3}(\tw_t)\tw_{x}\Big]\xi\,\dx\dt \right|
\\ 
 \quad \ls \epsilon \int_{0}^{T}\!\! E_{\Z}(t)\dt 
+ c_{\epsilon}\int_{0}^{T}\int_{0}^{L}\Big[a_{1}\tphi_{t}^{2}+a_{2}\tpsi_{t}^{2}+a_{3}\tw_{t}^{2}\Big] \dx\dt.
\end{split}
\end{align}

\noindent Let us estimate the kinetic energy in (\ref{Gest1}). Assumption $(g.1.)$ implies that 
\begin{align} \label{Gest2}\begin{split}
\int_{0}^{T}\int_{0}^{L}\Big[a_{1}\tphi_{t}^{2}&+a_{2}\tpsi_{t}^{2}+a_{3}\tw_{t}^{2}\Big] \dx\dt\\
&\leqslant c \int_{0}^{T}\int_{0}^{L}\Big[a_{1}G_{1}(\tphi_t)\tphi_{t} + a_{2}G_{2}(\tpsi_t)\tpsi_{t}+ a_{3}G_{3}(\tw_t)\tw_{t}\Big]\dx\dt.
\end{split}
\end{align}

\noindent Next, we estimate the source terms. Invoking Assumption $(f.3)$, we find a positive constant $c$ such that
\begin{align*}
\left|\int_{0}^{T}\int_{0}^{L} F_1(\tphi,\tpsi,\tw)\tphi_{x}\xi\, \dx \dt\right|
& \ls  c\int_{0}^{T}\int_{0}^{L} c(\nabla f_{1}) (|\,\varphi|+|\,\psi|+|\,w|)|\,\varphi_{x}|\, \dx \dt\\
& \ls \epsilon \int_{0}^{T}\!\! E_{\Z}(t)\dt + c_{\epsilon,\mathscr{B}}\mbox{l.o.t}(\tphi,\tpsi,\tw),
\end{align*}
where
$$
c(\nabla f_1)=1 + |\varphi^1|^{p-1} + |\varphi^2|^{p-1}
+ |\psi^1 |^{p-1} + | \psi^2 |^{p-1}
+ |w^1|^{p-1} + | w^2 |^{p-1}.
$$

\noindent The above implies that
\begin{align} \label{Fest11}\begin{split}
\left|\int_{0}^{T}\int_{0}^{L}\Big[ F_1(\tphi,\tpsi,\tw) \tphi_{x} + F_2(\tphi,\tpsi,\tw) \tpsi_{x} + F_3(\tphi,\tpsi,\tw) \tw_{x}\Big]\xi\,\dx\dt\right|\\
\quad \ls \epsilon \int_{0}^{T}\!\! E_{\Z}(t)\dt + c_{\epsilon,\mathscr{B}}\mbox{l.o.t}(\tphi,\tpsi,\tw).
\end{split}
\end{align}

\noindent Next, we combine (\ref{es11})-(\ref{Fest11}) with $(\ref{est1a})$. For sufficiently small $\epsilon >0$, we obtain
\begin{align}\label{est1}
\begin{split}
\int_{0}^{T}E_{\Z}(t)\dt 
\leqslant c\big[E_{\Z}(T)+E_{\Z}(0)\big]  
  + \frac{1}{2}\int_{0}^{T}\int_{L_{1}+\epsilon_{0}}^{L_{2}-\epsilon_{0}}E(t)\,\dx\dt 
  + c_{\mathscr{B}}\mbox{l.o.t}(\tphi,\tpsi,\tw)\\
 \quad +c_{\mathscr{B}} \int_{0}^{T}\int_{0}^{L}\Big[a_{1}G_{1}(\tphi_t)\tphi_{t} + a_{2}G_{2}(\tpsi_t)\tpsi_{t}+ a_{3}G_{3}(\tw_t)\tw_{t}\Big]\dx\dt.
\end{split}
\end{align} 
 
\noindent The next step is to estimate the integral of $E(\cdot)$ over the interval $[L_{1}+\epsilon,L_{2}+\epsilon]$. 
To this end, we consider the function
$[0,1]\ni \eta \in C^{\infty}(0,L)$ defined as follows
\begin{align*}              
&\eta(x)=\left \{
\begin{array}{l}
\displaystyle{\eta(x)=0,\,\, \mbox{if}\,\, x \in (0,L_{1})\cup(L_{2},L)} ,\\
\displaystyle{\eta(x)=1,\,\, \mbox{if}\,\, x \in (L_{1}+\epsilon_{0},L_{2}-\epsilon_{0}).}	                   														                   
\end{array}
\right. .
\end{align*}

\noindent We start multiplying the equations $(\ref{B4})$
by $\tphi \eta$, $\tpsi\eta$ and $\tw \eta$, respectively, 
and after integrate over $[0,T]\times[0,L]$, we add the kinetic energy 
$\int_{0}^{T}\int_{0}^{L}\bigl[\rho_{1}\tphi^{2} + \rho_{2}\tpsi_t^2 + \rho_{1}\tw_t^2\bigl]\eta\, \dx\dt$ to obtain 
\begin{align}\label{E2}
\begin{split}
\int_{0}^{T}&\int_{0}^{L}E(t)\eta\,\dx\dt\\
&=-\int_{0}^{L}\Big[\rho_{1}\tphi_{t}\tphi \!+\! \rho_{2}\tpsi_{t}\tpsi \!+\! \rho_{1}\tw_{t}\tw \Big]\eta \,\dx\Bigl|_{0}^{T} 
  \!+ 2\int_{0}^{T}\int_{0}^{L}\Big[\rho_{1}\tphi^{2} \!+\! \rho_{2}\tpsi_t^2 \!+\! \rho_{1}\tw_t^2\Big]\eta\, \dx\dt\\
&\quad - \int_{0}^{T}\int_{0}^{L}\Big[ k(\tphi_{x}+\tpsi+l  \tw)\tphi+b\tpsi_{x}\tpsi + k_{0}l (\tw_{x}-l \tphi)\tw\Big]\eta^{\prime}\,\dx\dt \\
&\quad - \int_{0}^{T}\int_{0}^{L}\Big[a_{1}G_{1}(\tphi_t)\tphi + a_{2}G_{2}(\tpsi_t)\tpsi+ a_{3}G_{3}(\tw_t)\tw\Big]\eta\,\dx\dt   \\
&\quad - \int_{0}^{T}\int_{0}^{L}\Big[ F_1(\tphi,\tpsi,\tw) \tphi + F_2(\tphi,\tpsi,\tw) \tpsi + F_3(\tphi,\tpsi,\tw) \tw\Big]\eta\,\dx\dt. 
\end{split}
\end{align} 

\noindent We shall estimate the right-hand side of (\ref{E2}). 
To this end, we repeat the pattern of estimates (\ref{es11})-(\ref{Fest11}) to find  
\begin{align}\label{est2}
\begin{split}
\int_{0}^{T}\int_{0}^{L}E(t)\eta\,\dx\dt 
&\leqslant c\big[E_{\Z}(T)+E_{\Z}(0)\big] + c_{\mathscr{B}}\mbox{l.o.t}(\tphi,\tpsi,\tw)\\
& \quad + c_{\mathscr{B}}\int_{0}^{T}\int_{0}^{L}\Big[a_{1}G_{1}(\tphi_t)\tphi_{t} + a_{2}G_{2}(\tpsi_t)\tpsi_{t}+ a_{3}G_{3}(\tw_t)\tw_{t}\Big]\dx\dt.
\end{split}
\end{align}  

\noindent Applying the estimate $(\ref{est2})$ above in $(\ref{est1})$, we obtain 
\begin{align}\label{est3}
\begin{split}
\int_{0}^{T}\!\! E_{\Z}(t)\dt 
&\leqslant c\big[E_{\Z}(T)+E_{\Z}(0)\big] + c_{\mathscr{B}}\mbox{l.o.t}(\tphi,\tpsi,\tw)\\
& \quad +c_{\mathscr{B}}\int_{0}^{T}\int_{0}^{L}\Big[a_{1}G_{1}(\tphi_t)\tphi_{t} + a_{2}G_{2}(\tpsi_t)\tpsi_{t}+ a_{3}G_{3}(\tw_t)\tw_{t}\Big]\dx\dt.
\end{split}
\end{align}  

\noindent Next, we estimate  damping terms on the right-hand side of $(\ref{est3})$.  
Multiply the equations (\ref{B4}) by $\varphi_t$, $\psi_t$, $w_t$, respectively. 
Then we find that
\begin{align} \label{EEnergy}
\begin{split}
\int_{0}^{T}\int_{0}^{L}\Big[a_{1}G_{1}(\tphi_t)\tphi_{t} + a_{2}G_{2}(\tpsi_t)\tpsi_{t}+ a_{3}G_{3}(\tw_t)\tw_{t}\Big]\dx\dt - E_{\Z}(0) + E_{\Z}(T)
\\ 
= - \int_{0}^{T}\int_{0}^{L}\Big[ F_1(\tphi,\tpsi,\tw) \tphi_t  + F_2(\tphi,\tpsi,\tw) \tpsi_t + F_3(\tphi,\tpsi,\tw) \tw_t \Big] \dx \dt.
\end{split}
\end{align}

\noindent Based on  estimate (\ref{Fest11}), we obtain  
\begin{align} \label{Fest112}
\begin{split}
\left|\int_{0}^{T}\int_{0}^{L}\Big[F_1(\tphi,\tpsi,\tw) \tphi_{t} + F_2(\tphi,\tpsi,\tw) \tpsi_{t} + F_3(\tphi,\tpsi,\tw) \tw_{t}\Big]\dx\dt\right|\\
\quad \ls \epsilon \int_{0}^{T}\!\! E_{\Z}(t)\dt + c_{\epsilon,\mathscr{B}}\mbox{l.o.t}(\tphi,\tpsi,\tw).
\end{split}
\end{align}

\noindent Using both $(\ref{EEnergy})$ and $(\ref{Fest112})$, we find 
\begin{align} \label{EEnergy1}
\begin{split}
\int_{0}^{T}\int_{0}^{L}  \Big[a_{1}G_{1}(\tphi_t)\tphi_{t} + a_{2}G_{2}(\tpsi_t)\tpsi_{t}+ a_{3}G_{3}(\tw_t)\tw_{t}\Big]\dx\dt
\\ 
 \leqslant E_{\Z}(0) - E_{\Z}(T) + \epsilon \int_{0}^{T}\!\! E_{\Z}(t)\dt + c_{\epsilon,\mathscr{B}}\mbox{l.o.t}(\tphi,\tpsi,\tw).
\end{split}
\end{align}

\noindent We return to (\ref{est3}) and obtain by use of (\ref{EEnergy1}), with $\epsilon>0$ small enough, the following estimate 
\begin{equation} \label{resumo}
\int_{0}^{T} \!\! E_{\Z}(t) \dt \leqslant  (c-c_{\mathscr{B}}) E_{\Z}(T)+(c+c_{\mathscr{B}})E_{\Z}(0)+ c_{\mathscr{B}}\mbox{l.o.t}(\tphi,\tpsi,\tw).
\end{equation}

\noindent The next step is to estimate the energy $E_{\Z}(\cdot)$. To this end, we use the multipliers $\varphi_t$, $\psi_t$, $\tw_t$ for $(\ref{B4})$. 
Then, after integration, we find
\begin{align}\label{123}
\begin{split} 
T E_{\Z}(T)
 =  \int_{0}^{T}\!\! E_{\Z}(t)\dt
- \int_{0}^{T}\int_{s}^{T}\int_{0}^{L}\Big[a_{1}G_{1}(\tphi_t)\tphi_{t} + a_{2}G_{2}(\tpsi_t)\tpsi_{t}+ a_{3}G_{3}(\tw_t)\tw_{t}\Big]\dx\dt{\rm d}s  \\
- \, \int_{0}^{T}\int_{s}^{T}\int_{0}^{L} \Big[F_1(\tphi,\tpsi,\tw) \tphi_{t} + F_2(\tphi,\tpsi,\tw) \tpsi_{t} + F_3(\tphi,\tpsi,\tw) \tw_{t}\Big]\dx\dt{\rm d}s.
\end{split}
\end{align}

\noindent Now, the forcing assumptions give
\begin{align*}
\begin{split}
\int_{0}^{L}\Big[F_1(\tphi,\tpsi,\tw) \tphi_{t} + F_2(\tphi,\tpsi,\tw) \tpsi_{t} + F_3(\tphi,\tpsi,\tw) \tw_{t}\Big]\dx
\ls \frac{1}{T}E_{\Z}(t)+ c_{T,\mathscr{B}}\mbox{l.o.t}(\tphi,\tpsi,\tw).
\end{split}
\end{align*}

\noindent The above implies
\begin{align*} 
\int_{0}^{T}\int_{s}^{T}\int_{0}^{L} \Big[F_1(\tphi,\tpsi,\tw) \tphi_{t} + F_2(\tphi,\tpsi,\tw) \tpsi_{t} + F_3(\tphi,\tpsi,\tw) \tw_{t}\Big]\dx\dt{\rm d}s\\
\ls \int_{0}^{T}E_{\Z}(t)\dt+ c_{T,\mathscr{B}}\mbox{l.o.t}(\tphi,\tpsi,\tw).
\end{align*}

\noindent We combine the above estimates with (\ref{123})
\begin{equation} \label{A21}
T E_{\Z}(T) \ls 2 \int_{0}^{T}E_{\Z}(t)\dt + c_{T,\mathscr{B}}\mbox{l.o.t}(\tphi,\tpsi,\tw). 
\end{equation}

\noindent Next, we using estimate (\ref{resumo}) in (\ref{A21}) we arrive at 
$$
T E_{\Z}(T) \leq 2(c-c_{\mathscr{B}}) E_{\Z}(T)+ 2(c+c_{\mathscr{B}})E_{\Z}(0)+ c_{T,\mathscr{B}}\mbox{l.o.t}(\tphi,\tpsi,\tw).
$$

\noindent Taking $T> 4c$, we find
$$
E_{\Z}(T) \leq \frac{2(c+c_{\mathscr{B}})}{T-2(c-c_{\mathscr{B}})}E_{\Z}(0)+ c_{T,\mathscr{B}}\mbox{l.o.t}(\tphi,\tpsi,\tw). $$

\noindent Using standard stabilization arguments, we obtain the existence of positive constants $c_{1}= c_\mathscr{B}$ and $ \omega = \omega_\mathscr{B}$ such that
$$
\|\Z(t)\|^{2}_{H} \leq c_{1} \|\Z(0)\|^{2}_{H} e^{-\omega t} + c_{1}\sup_{\sigma \in [0,t]}\Bigl[ \| \varphi(\sigma) \|_{2p}^2
+ \| \psi (\sigma)\|_{2p}^2 + \| w (\sigma)\|_{2p}^2 \Bigl].
$$
Therefore, the inequality $(\ref{SI})$ holds with  
$X = [H^{1}_{0}(0,L)]^{3}$, $Y = [L^{2}(0,L)]^{3}$, $b(t) =c_{1} e^{-\omega t}$, $c(t) = c_{1}$ 
and with compact semi-norm 
$$
[(\tphi,\tpsi,\tw) ]_{X}^{2} = \|\tphi\|_{2p}^2+\|\tpsi\|_{2p}^2+\|\tw\|_{2p}^2.
$$
\qed

\subsubsection{Proof of the main result - Theorem \ref{Main}}

\noindent \textit{Proof of Theorem} \ref{Main}: From Proposition \ref{P-Gradient} and Proposition \ref{P-Quasi}, 
we have that $(H,S(t))$  is a gradient quasi-stable system. Moreover, by inequality (\ref{energia-domina}) one can see that the
Lyapunov function defined as the energy $\mathscr{E}_{\Z}(\cdot)$ satisfies the following: 
(i) $\Phi(\cdot)$ is bounded from above on any bounded set and 
(ii) the $\Phi(R) = \big\{ \normalfont{\Z}\in H\,|\, \Phi(\Z)\ls R \big\}$ is bounded for every $R$. 
To conclude the proof, we note that if $\Z \in \mathscr{N}$, 
then $\Z=(\varphi,\psi,w,0,0,0)$ solves the stationary problem
\begin{align} \label{PE01}
\left\{ 
\begin{array}{rrr}
\begin{split}
- k(\varphi_{x}+\psi+l  w)_{x} - k_{0}l  (w_{x}-l \varphi) + f_1(\varphi,\psi,w) =0, &  \\
- b\psi_{xx}+k(\varphi_{x} + \psi+l  w)+ f_2(\varphi,\psi,w) = 0, &  \\
- k_{0}(w_{x}-l \varphi)_{x} + kl (\varphi_{x}+\psi+l  w) + f_3(\varphi,\psi,w) =0. &   
\end{split}
\end{array}
\right.
\end{align}

Multiplying in $L^{2}(0,L)$ the equations in $(\ref{PE01})$ by $(\varphi,\psi,w)$, we find
\begin{align*}
b\|\psi_{x}\|^{2} &+k\|\varphi_{x} + \psi+l w\|^{2}+k_{0}\|w_{x} + l\varphi\|^{2} =  - \int_{0}^{L}\nabla F(u,v,w) \! \cdot \! (u,v,w)  \dx.
\end{align*}
Now, we use Assumption (f.1) to show
$$
\left[ 1 - \frac{2\alpha\beta L^2}{\pi^2} \right] \Big[  \| \varphi_{x} \|^{2} + \|\psi_{x}\|^{2} + \| w_x \|^{2} \Big] \ls   2 \beta c_F L. 
$$
Therefore, the set of stationary solutions $\mathscr{N}$ is bounded. This completes the proof.

\appendix
\section*{Appendix: Well-possednes for overdetermined wave equations}
\renewcommand{\theTheorem}{A.\arabic{Theorem}}
\renewcommand{\theequation}{A.\arabic{equation}}
\renewcommand{\thesection}{A}

In this appendix we will guarantee the well-posedness for the system presented in \eqref{P} with overdetermined condition.

\begin{Theorem} \label{apen1}
Let $L>0$ and $T>0$ large enough. 
If the Problem \eqref{P}-\eqref{initial} satisfies \eqref{delta-i}-\eqref{Ei} with supplementary condition 
	\begin{equation} \label{over}
	(u_1,\cdots,u_n) = (0,\cdots,0)\ \mbox{in} \  \omega \times [0,T],  
	\end{equation}
with $\omega \subset [0,L_0]$ as in \eqref{omega}. 
Then, the overdetermined problem is well-posed and generates a strongly continuous semigroup over the Hilbert space
\begin{align*}
H_{\omega_j}(M_{i}) \equiv \left\{
 (u_1,\cdots,u_n,v_1,\cdots,v_n)  \left|
\begin{array}{l}
  u_{i} \in H^1_0 (M_{i}),\ v_{i} \in L^{2} (M_{i}), \\
  u_{i}=v_{i}=0 \ \mbox{in} \ \omega, i=1,\cdots,n
\end{array}
\right.
\right\},
\end{align*}
where $M_i=([0,L],g_i)$.
\end{Theorem}

\begin{proof}
First, for the state vector $\Z(t) \equiv (u_1,\cdots, u_n, \partial_t u_1 ,\cdots, \partial_t u_n)^{\top}$, 
the Problem \eqref{P}-\eqref{initial} is equivalent to the following vectorial Cauchy problem
\begin{equation} \label{sis}
\partial_t \Z(t) + A \Z(t)=\mathscr{F}(\Z(t)), \ \ \Z(0)= (u^0_1,\cdots, u^0_n, u^1_1,\cdots, u^1_n)^{\top},
\end{equation}
with operators defined by 
\begin{align*}
A=\begin{bmatrix}
0 & -I|_{V_2(M_i)}  \\
B|_{V_1(M_i)} & 0  \\
\end{bmatrix}, \ 
B=\begin{bmatrix}
-\Delta_1 & 0 & \cdots & 0  \\
0 & -\Delta_2  & \cdots & 0  \\
\vdots & \vdots  &  \cdots &  \vdots \\
0& 0 & 0 & -\Delta_n \\
\end{bmatrix}, 
\ I=(\delta_{ij})_{n\times n} 
\end{align*}
and 
$$
\mathscr{F}(\Z)=\begin{bmatrix}
0  \\
\vdots  \\
0 \\
f_1(u_1,\cdots,u_n) \\
\vdots  \\
f_n(u_1,\cdots,u_n) \\
\end{bmatrix}.
$$
The domain of operator $A$ is defined by $D(A)=[V_1(M_i)]^{n}\times [V_2(M_i)]^{n}$ where
\begin{align*}
V_1(M_i)& \equiv \bigl\{ (u_1,\cdots,u_n) \in D(B) \ | \ (u_1,\cdots,u_n)=(0,\cdots,0) \ \text{in $\omega$} \bigl\}, \\
V_2(M_i)& \equiv \bigl\{ (v_1,\cdots,v_n) \in D(B^{\frac{1}{2}}) \ | \ (v_1,\cdots,v_n)=(0,\cdots,0)  \ \text{in $\omega$} \bigl\}
\end{align*}
and
$$
D(B)=D(-\Delta_1)\times \cdots \times D(-\Delta_n)=H^2(M_1)\cap H^1_0(M_1) \times \cdots H^2(M_n)\cap H^1_0(M_n).
$$
 
\noindent The finite energy space for (\ref{sis}) is the Hilbert space defined by  
$$
H_{\omega}(M_i) \equiv 
\bigl\{(u_1,\cdots,u_n,v_1,\cdots,v_n) \in {H}(M_i) \ | \ u_i=v_i=0  \ \text{in} \  \omega, \ \forall i=1,\cdots,n\bigl\},
$$
where
$$
H(M_i)\equiv [D(B^{\frac{1}{2}})]^{n} \times L(M_i) \ \mbox{and} \  L(M_i)\equiv L^2(M_1) \times \cdots \times L^2(M_n) .
$$

Using classical semigroup theory, one can establish existence and uniqueness of a solution to the Cauchy problem \eqref{sis}. 
Moreover, the solution operator generates a strongly continuous semigroup
$$
T_{M_i}(t): {H}_w(M_i) \to {H}_w(M_i), 
$$
defined by
$$
(u^0_1,\cdots,u^0_n,u^1_1,\cdots,u^1_n) \mapsto (u_1(t),\cdots,u_n(t),\partial_t u_1(t),\cdots,\partial_t u_n(t)), \ \ t\ge 0,
$$ 
where $(u_1,\cdots,u_n,\partial_t u_1,\cdots,\partial_t u_n)$ is the weak solution corresponding to  the initial data 
$$
(u^0_1,\cdots,u^0_n,u^1_1,\cdots,u^1_n).
$$
In addition, $\{T_{M_i}(t)\}_{t\ge 0 }$ is also strongly continuous semigroup on ${H}_\omega(M_i)$ 
satisfying the compatibility condition \eqref{over}.

\begin{Remark} \label{obskey} It is not difficult to show that if \eqref{over} is fulfilled, then we also have
		\begin{equation*}
	(\partial_t u_1,\cdots,\partial_t u_n) = (0,\cdots,0) \quad \hbox{in} \; \omega \times [0,T].
	\end{equation*}
\end{Remark}

\end{proof}


\paragraph*{Funding:} The first author is partially supported by CNPq grant 312529/2018-0.  
The third author is supported by INCTMat-CAPES grant 88887.507829/2020-00.


\paragraph{Email addresses}  

\begin{itemize} 
	\item T. F. Ma: matofu@mat.unb.br 
	
	\item R. N. Monteiro: monteirorn@uel.br
	
	\item P. N. Seminario-Huertas: pseminariohuertas@gmail.com
	
\end{itemize}

\end{document}